\documentclass[a4paper, 12pt, leqno]{article}
\usepackage{amsmath}
\usepackage{amsfonts}
\usepackage{amssymb}
\usepackage{amsthm}
\usepackage{mathrsfs}
\usepackage{graphicx}
\usepackage{eufrak}
\usepackage[colorlinks, linkcolor=red, anchorcolor=blue, citecolor=green]{hyperref} 
\usepackage{color}
\usepackage{enumerate}
\usepackage{authblk} 
\usepackage{marvosym} 

\numberwithin{equation}{section}
\newtheorem{theorem}{Theorem}
\numberwithin{theorem}{section}
\newtheorem*{theorem*}{Theorem} 
\newtheorem{proposition}[theorem]{Proposition}

\newtheorem{definition}[theorem]{Definition}

\newtheorem{lemma}[theorem]{Lemma}
\newtheorem{notation}[theorem]{Notation}

\newtheorem{remark}[theorem]{Remark}
\newtheorem{corollary}[theorem]{Corollary}

\allowdisplaybreaks[2]

\makeatletter
\newcommand{\subjclass}[2][1991]{%
  \let\@oldtitle\@title%
  \gdef\@title{\@oldtitle\footnotetext{#1 \emph{Mathematics subject classification:} #2}}%
}
\newcommand{\keywords}[1]{%
  \let\@@oldtitle\@title%
  \gdef\@title{\@@oldtitle\footnotetext{\emph{Key words and phrases.} #1.}}%
}
\makeatother

\begin{document}

\title{Pointwise and uniform bounds for functions of the Laplacian on non-compact symmetric spaces}

\renewcommand\Affilfont{\itshape\small}

\author[1]{Yulia Kuznetsova}
\affil[1]{Universit\'e de Franche Comt\'e, CNRS, LmB (UMR 6623), F-25000 Besançon, France}

\author[1,2]{Zhipeng Song}
\affil[2]{Ghent University, Department of Mathematics: Analysis, Logic and Discrete Mathematics, 9000 Ghent, Belgium}


{
\let\thefootnote\relax\footnotetext{\Letter\ Yulia Kuznetsova: yulia.kuznetsova@univ-fcomte.fr; Zhipeng Song: zhipeng.song@univ-fcomte.fr}
}


{
}

\maketitle

\abstract{Let $L$ be the distinguished Laplacian on the Iwasawa $AN$ group associated with a semisimple Lie group $G$. Assume $F$ is a Borel function on $\mathbb{R}^+$. We give a condition on $F$ such that the kernels of the functions $F(L)$ are uniformly bounded. This condition involves the decay of $F$ only and not its derivatives. 
By a known correspondence, this implies pointwise estimates for a wide range of functions of the Laplace-Beltrami operator on symmetric spaces. In particular, when $G$ is of real rank one and $F(x)={\rm e}^{it\sqrt x}\psi(\sqrt x)$, our bounds are sharp.}

\newcommand{\ad}{\operatorname{ad}} 
\newcommand{\Ad}{\operatorname{Ad}} 
\newcommand{\id}{\operatorname{id}}
\newcommand{\Id}{\operatorname{Id}} 
\newcommand{\supp}{\operatorname{supp}} 
\newcommand{\e}{{\rm e}} 
\newcommand{\real}{\operatorname{Re}} 
\newcommand{\imaginary}{\operatorname{Im}} 
\newcommand{\grad}{\operatorname{\text{grad}}} 
\newcommand{\divergence}{\operatorname{\text{div}}} 
\newcommand{\ch}{\operatorname{ch}} 
\newcommand{\tr}{\operatorname{tr}} 
\newcommand{\End}{\operatorname{End}} 
\newcommand{\Hom}{\operatorname{Hom}} 
\newcommand{\Der}{\operatorname{Der}} 
\newcommand{\Aut}{\operatorname{Aut}} 
\newcommand{\Int}{\operatorname{Int}} 
\newcommand{\linearspan}{\operatorname{span}} 
\newcommand{\rank}{\operatorname{rank}} 

\newcommand{\mfa}{\operatorname{\mathfrak{a}}}
\newcommand{\mfb}{\operatorname{\mathfrak{b}}}
\newcommand{\mfc}{\operatorname{\mathfrak{c}}}
\newcommand{\mfd}{\operatorname{\mathfrak{d}}}
\newcommand{\mfe}{\operatorname{\mathfrak{e}}}
\newcommand{\mff}{\operatorname{\mathfrak{f}}}
\newcommand{\mfg}{\operatorname{\mathfrak{g}}}
\newcommand{\mfh}{\operatorname{\mathfrak{h}}}
\newcommand{\mfi}{\operatorname{\mathfrak{i}}}
\newcommand{\mfj}{\operatorname{\mathfrak{j}}}
\newcommand{\mfk}{\operatorname{\mathfrak{k}}}
\newcommand{\mfl}{\operatorname{\mathfrak{l}}}
\newcommand{\mfm}{\operatorname{\mathfrak{m}}}
\newcommand{\mfn}{\operatorname{\mathfrak{n}}}
\newcommand{\mfo}{\operatorname{\mathfrak{o}}}
\newcommand{\mfp}{\operatorname{\mathfrak{p}}}
\newcommand{\mfq}{\operatorname{\mathfrak{q}}}
\newcommand{\mfr}{\operatorname{\mathfrak{r}}}
\newcommand{\mfs}{\operatorname{\mathfrak{s}}}
\newcommand{\mft}{\operatorname{\mathfrak{t}}}
\newcommand{\mfu}{\operatorname{\mathfrak{u}}}
\newcommand{\mfv}{\operatorname{\mathfrak{v}}}
\newcommand{\mfw}{\operatorname{\mathfrak{w}}}
\newcommand{\mfx}{\operatorname{\mathfrak{x}}}
\newcommand{\mfy}{\operatorname{\mathfrak{y}}}
\newcommand{\mfz}{\operatorname{\mathfrak{z}}}

\newcommand{\mfA}{\operatorname{\mathfrak{A}}}
\newcommand{\mfB}{\operatorname{\mathfrak{B}}}
\newcommand{\mfC}{\operatorname{\mathfrak{C}}}
\newcommand{\mfD}{\operatorname{\mathfrak{D}}}
\newcommand{\mfE}{\operatorname{\mathfrak{E}}}
\newcommand{\mfF}{\operatorname{\mathfrak{F}}}
\newcommand{\mfG}{\operatorname{\mathfrak{G}}}
\newcommand{\mfH}{\operatorname{\mathfrak{H}}}
\newcommand{\mfI}{\operatorname{\mathfrak{I}}}
\newcommand{\mfJ}{\operatorname{\mathfrak{J}}}
\newcommand{\mfK}{\operatorname{\mathfrak{K}}}
\newcommand{\mfL}{\operatorname{\mathfrak{L}}}
\newcommand{\mfM}{\operatorname{\mathfrak{M}}}
\newcommand{\mfN}{\operatorname{\mathfrak{N}}}
\newcommand{\mfO}{\operatorname{\mathfrak{O}}}
\newcommand{\mfP}{\operatorname{\mathfrak{P}}}
\newcommand{\mfQ}{\operatorname{\mathfrak{Q}}}
\newcommand{\mfR}{\operatorname{\mathfrak{R}}}
\newcommand{\mfS}{\operatorname{\mathfrak{S}}}
\newcommand{\mfT}{\operatorname{\mathfrak{T}}}
\newcommand{\mfU}{\operatorname{\mathfrak{U}}}
\newcommand{\mfV}{\operatorname{\mathfrak{V}}}
\newcommand{\mfW}{\operatorname{\mathfrak{W}}}
\newcommand{\mfX}{\operatorname{\mathfrak{X}}}
\newcommand{\mfY}{\operatorname{\mathfrak{Y}}}
\newcommand{\mfZ}{\operatorname{\mathfrak{Z}}}

\newcommand{\mA}{\operatorname{\mathscr{A}}}
\newcommand{\mB}{\operatorname{\mathscr{B}}}
\newcommand{\mC}{\operatorname{\mathscr{C}}}
\newcommand{\mD}{\operatorname{\mathscr{D}}}
\newcommand{\mE}{\operatorname{\mathscr{E}}}
\newcommand{\mF}{\operatorname{\mathscr{F}}}
\newcommand{\mG}{\operatorname{\mathscr{G}}}
\newcommand{\mH}{\operatorname{\mathscr{H}}}
\newcommand{\mI}{\operatorname{\mathscr{I}}}
\newcommand{\mJ}{\operatorname{\mathscr{J}}}
\newcommand{\mK}{\operatorname{\mathscr{K}}}
\newcommand{\mL}{\operatorname{\mathscr{L}}}
\newcommand{\mM}{\operatorname{\mathscr{M}}}
\newcommand{\mN}{\operatorname{\mathscr{N}}}
\newcommand{\mO}{\operatorname{\mathscr{O}}}
\newcommand{\mP}{\operatorname{\mathscr{P}}}
\newcommand{\mQ}{\operatorname{\mathscr{Q}}}
\newcommand{\mR}{\operatorname{\mathscr{R}}}
\newcommand{\mS}{\operatorname{\mathscr{S}}}
\newcommand{\mT}{\operatorname{\mathscr{T}}}
\newcommand{\mU}{\operatorname{\mathscr{U}}}
\newcommand{\mV}{\operatorname{\mathscr{V}}}
\newcommand{\mW}{\operatorname{\mathscr{W}}}
\newcommand{\mX}{\operatorname{\mathscr{X}}}
\newcommand{\mY}{\operatorname{\mathscr{Y}}}
\newcommand{\mZ}{\operatorname{\mathscr{Z}}}

\newcommand{\R}{\operatorname{\mathbb{R}}}
\newcommand{\C}{\operatorname{\mathbb{C}}}
\newcommand{\Z}{\operatorname{\mathbb{Z}}}
\newcommand{\N}{\operatorname{\mathbb{N}}}
\newcommand{\Q}{\operatorname{\mathbb{Q}}}
\newcommand{\F}{\operatorname{\mathbb{F}}}
\newcommand{\K}{\operatorname{\mathbb{K}}}
\newcommand{\X}{\operatorname{\mathbb{X}}}

\def\abso #1{ \left| #1 \right| } 
\def\set #1{ \left\{ #1 \right\} } 
\def\norm #1{ \left\| #1 \right\| } 
\def\bracket #1{ \left( #1 \right) } 
\def\Bracket #1{ \left[ #1 \right] } 
\def\ip #1{ \left\langle #1 \right\rangle } 

\newcommand{\HCc}{\operatorname{\mathbf{c}}} 
\newcommand{\IDO}{\operatorname{\mathbf{D}}}
\newcommand{\LB}{{F(\Delta_\rho)}}
\newcommand{\La}{{F(L)}} 
\newcommand{\HCb}{\operatorname{\mathbf{b}}} 
\newcommand{\CA}{{\text{Cl($A^+$)}}} 
\newcommand{\Ca}{{\text{Cl($\mfa^+$)}}} 
\newcommand{\wall}{{\text{wall($\mfa^+$)}}} 
\let\a\alpha


\section{Introduction}\label{intro}

An important part of the classical harmonic analysis is the study of multipliers. Let $k$ be a tempered distribution on $\R^n$ and $m=\mF k$ be its Fourier transform. The celebrated H\"omander-Mikhlin multiplier theorem \cite{hormander1960estimates} asserts that the convolution operator $T\cdot:=\cdot *k$ is bounded on $L^p(\R^n)$ $(1< p<\infty)$ if the symbol $m$ satisfies the Mikhlin type condition:
\[
\sup_{\xi\in \R^n}\abso{\xi}^{|\alpha|}\abso{\partial_\xi^\alpha m(\xi)}\le C< +\infty, \quad \forall |\alpha|\le \Bracket{n/2}+1.
\]


Much of this theory makes sense on Lie groups, if we restrict our attention to spherical multipliers. Let $G$ be a semisimple, connected, and noncompact Lie group. Assuming $G$ has a finite center, it has an Iwasawa decomposition $G=KAN$, and there exists the Harish-Chandra transform $\mH$, also called the spherical transform, taking $K$-biinvariant functions into functions on $\mfa_{\C}^*$. Here $\mfa_{\C}^*$ is the dual of complexification $\mfa_{\C}$ of the Lie algebra $\mfa$ of $A$. Let $k$ be a $K$-biinvariant tempered distribution on $G$ and $\mH k$ its Harish-Chandra transform. Then the operator $T\cdot=\cdot *k$ is a spherical multiplier on the symmetric space $S=G/K$, of kernel $k$ and symbol $m=\mH k$. 
We also refer reader to \cite{gangolli1988harmonic,helgason1962differential,helgason2022groups}.

The conditions on $m$ are however very different from the Euclidean case: by a famous result of Clerc and Stein \cite{clerc1974lp}, $m$ must be holomorphic and bounded in an open tube around the real part $\mfa^*$ of $\mfa_{\C}^*$, in order to make $T$ bounded on $L^p(S)$ for $1<p<2$. Conversely, if this condition is satisfied, then a Mikhlin-type bound completes the picture to give a sufficient condition. 
This was proved by Clerc and Stein for $G$ complex, Stanton and Tomas \cite{stanton1978expansions}  in the rank-one case, and finally by Anker~\cite{anker1990lp} in the case of a higher rank.

From the PDEs point of view, the most important class of multipliers is of those generated by the Laplace operator. We state the problem in the case of symmetric spaces directly. To the Riemannian structure on $S$, there corresponds canonically its Laplace--Beltrami operator $\Delta$. It is self-adjoint on $L^2(S)$. Consequently, by spectral theorem, every bounded Borel function $F$ on $\R$ defines a bounded operator on $L^2(S)$:
\[
F(\Delta)=\int_0^\infty F(\xi) dE(\xi),
\] 
where $E$ is the spectral measure of $\Delta$. Moreover, being left-invariant, $F(\Delta)$ acts as a convolution on the right with a kernel $k_{F(\Delta)}$ (a priori defined in the distributional sense). 
An interesting question is, under what conditions on $F$, the operator $F(\Delta)$ is also bounded from $L^p(S)$ to $L^q(S)$?
This requires estimating $L^p$ norms of the kernel, which often requires finding its pointwise bounds. 

Such pointwise bounds are known for a variety of functions $F$ and are obtained case by case. The heat kernel corresponding to $\exp(-t|\Delta|)$ is very well studied, see the book \cite{var-book} or a more recent survey \cite{anker-survey}; the latest results \cite{anker-heat-solutions} describe the asymptotics of solutions of the heat equation at large time. Moreover, bounds are available (most of them by Anker and coauthors) 
for:
 the resolvents $(z-\Delta)^{-s}$ \cite{anker1992sharp};
 the Poisson kernel corresponding to $\exp(-t\sqrt{|\Delta|})$ \cite{anker1992sharp};
 the Schr\"odinger kernel of the semigroup $\exp(it\Delta)$ \cite{anker-schr};
 oscillating functions of the type $(|\Delta|+q)^{-\tau} (|\Delta|+\tilde q)^{\,\sigma} \exp(it\sqrt{|\Delta|+q})$ with different exponents $\tau$ and~$\sigma$, see \cite{tataru,anker-hyperbolic,anker-damek} in rank one and recently \cite{hassani,anker-hongwei} in higher ranks.

In this paper we obtain pointwise estimates for $F(\Delta)$ with a very general function $F$, on spaces of arbitrary rank. More precisely, we get estimates for the shifted Laplacian $\Delta_\rho = -\Delta - |\rho|^2$, which has no spectral gap; $\rho$ is a fundamental linear functional on $\mfa$, see Section 2. Our main theorem (Corollary \ref{main-LB} in the text) reads:
\begin{theorem*}
    Let $F$ be a Borel function on $\R^+$ such that 
    \begin{equation}\label{integral-intro}
I_F =    \int_{\mfa^*} \abso{F(\abso{\lambda}^2)} (1+|\lambda|)^{7(n-l)+1} d\lambda < \infty.
    \end{equation}
    Then for every $x\in G$
    \[
    |k_{F(\Delta_\rho)}(x)| \le C I_F\, e^{-\rho(H(x))},
    \]
    where $C$ is a constant depending on $G$ only, and $H(x)$ is the radial part of $x$ in the Iwasawa decomposition.
\end{theorem*}
See Section 2 for definition of the dimensions $n$, $l$.

This is equivalent to Theorem \ref{uniform estimate} below. We show also that the upper limit of $e^\rho k_{F(\Delta_\rho)}$ at infinity is bounded by a similar integral with the power $n-l$ instead of $7(n-l)+1$.

It is clear that for $F$ positive, one cannot hope for a much better bound: it is sufficient to evaluate $k_{F(\Delta_\rho)}$ at identity. 
But for oscillating functions \eqref{integral-intro} can be close to optimal as well: we show in Theorem \ref{lower-rank1} that for $F(|\lambda|^2)=\e^{it|\lambda|}\psi(|\lambda|)$ and in rank one, the lower limit of $e^\rho k_{F(\Delta_\rho)}$ at infinity is never zero and is estimated from below by a similar integral. 
This generalizes results of~\cite{akylzhanov2022norms} obtained for hyperbolic spaces.

From known estimates 
of spherical functions,
it is easy in fact to obtain bounds for the kernel of the type $|k_{F(\Delta_\rho)}(x)| \lesssim (1+|H(x)|)^d e^{-\rho(H(x))}$, see discussion at the end of Section 4. Our result is that it is possible to remove the polynomial factor in $H(x)$.
To achieve this, we refine the asymptotic estimates of Harish-Chandra for spherical functions $\varphi_\lambda$. In the technical Section 3, we obtain explicit constants controlling the growth in $\lambda$ in these estimates.

A polynomial factor in $H(x)$ is often non-significant in the analysis of the Laplace-Beltrami operator. But it becomes essential in concern with the following operator $L$ which has attracted much attention. 
We can define it as $L = \delta^{-1/2}\tau \Delta_\rho \tau\delta^{1/2}$, where $\tau$ is the inversion: $(\tau f)(g) = f(g^{-1})$, and $\delta$ is the modular function of $S$. This makes $L$ act as a convolution on the left, and the kernel of $F(L)$ is linked to that of $F(\Delta_\rho)$ by $k_{F(L)} = \delta^{-1/2} k_{F(\Delta_\rho)}$. Importantly, $L$ can be written as the sum of squares $L=\sum X_j^2$ of vector fields generating the Lie algebra of $S$, in a full analogy with the Euclidean case; whereas $\Delta$ has a necessary linear term in addition, due to non-unimodularity of $S$.




It turns out that the theory of $L$ is very different from that of $\Delta$: Hebisch \cite{hebisch1993subalgebra} was first to prove that if $G$ is complex, then $F(L)$ is bounded on $L^p(S)$ for $1\le p<\infty$ as soon as  $F$ satisfies the Mikhlin condition. In other words, no holomorphy is needed. 
In 1994 Cowling et al.~\cite{cowling1994spectral} extended this result (with other constants replacing $\Bracket{n/2}+1$) to all real Lie groups; they gave a sufficient condition on $F$ such that $F(L)$ is bounded on $L^p$ for $1<p<\infty$ and is of weak type $(1,1)$.
Sikora \cite{sikora2002spectral} obtained a stronger optimal order of differentiability, which is $1/2$ smaller than the result of Cowling et al. Time-dependent bounds on $L^1$ for special oscillating functions were obtained by Gadzi\'nski \cite{gadzinski}.

We are interested in continuing this comparison. Thus, our Theorem \ref{uniform estimate} is stated in terms of $L$: the kernel $k_{F(L)}$ is uniformly bounded as soon as the integral \eqref{integral-intro} converges.

We show finally that in rank one, this estimate is best possible even for oscillating functions of the type $F(x) = \exp(it\sqrt x) \psi(\sqrt x)$: the uniform norm of the kernel $k_{F(L)}$ is bounded but does not decay at large time $t$. This generalizes the results of \cite{Muller2007axb, akylzhanov2022norms} for $ax+b$ groups (more precisely, their parts concerning uniform norms); we note also a related subsequent result of M\"uller and Vallarino \cite{Muller2010Damek-Ricci} on Damek-Ricci spaces, a class not covering all rank one symmetric spaces but containing many non-symmetric ones.



\medskip
{\bf Acknowledgements.} This work is supported by the EIPHI Graduate School (contract ANR-17-EURE-0002). The first author is also supported by the ANR-19-CE40-0002 grant of the French National Research Agency (ANR). The second author is also supported by the Methusalem programme of the Ghent University Special Research Fund (BOF), grant number 01M01021.

\section{Notations and preliminaries}

Our main interest is in semisimple groups, but the proofs involve results on reductive groups. This class, also called the class of Harish-Chandra, is defined differently from one author to another. We stick to the following definition of Gangolli and Varadarajan \cite{gangolli1988harmonic}, and term it class $\mathcal {H}$ accordingly:
\begin{definition}\label{class H}
A real Lie group $G$ with its Lie algebra $\mfg$ is in class $\mathcal{H}$ if
$\mfg$ is reductive;
     $G$ has only finitely many connected components;
     $\Ad (G) \subset \Int (\mfg_{\C})$;
    and the analytic subgroup of $G$ with Lie algebra $\Bracket{\mfg, \mfg}$ has finite center.
\end{definition}

Every semisimple, connected Lie group with finite center is in class $\mathcal H$.

\subsection{Iwasawa decomposition}\label{Iwasawa}
Let $G$ be a Lie group in class $\mathcal H$ with the Lie algebra $\mfg$.
Let $K$ be a maximal compact subgroup of $G$ and $\mfk$ the Lie algebra of $K$. We have the Cartan decomposition $\mfg = \mfk\oplus \mfp$ and the involution $\theta$ acting as 1 on $\mfk$ and -1 on $ \mfp$. 
Let $\mfa$ be a maximal abelian subspace of $\mfp$, this is, $\mfa$ is a Lie subalgebra of $\mfg$ such that $\mfa\subset \mfp$. 

We denote by $\mfa^*$ the real dual space of $\mfa$ and $\mfa_{\C}^*$ the dual of its complexification $\mfa_{\C}$. 

Let $\Sigma\subseteq  \mfa^* $ be the set of restricted roots of $(\mfg,\mfa)$. We have the restricted root space decomposition 
\begin{align}\label{decRoot}
	 \mfg= \mfg_0 \oplus \bigoplus_{\alpha\in \Sigma} \mfg_\alpha
\end{align}
where $ \mfg_\alpha=\{X\in  \mfg|(\ad H)X=\alpha(H)X,\ \forall H\in  \mfa\}$. Setting $\mfn=\bigoplus_{\alpha\in \Sigma^+}\mfg_{\alpha}$, we obtain the Iwasawa decomposition $\mfg= \mfk\oplus \mfa\oplus \mfn$ with $\mfa$ abelian and $\mfn$ nilpotent.

Denote by $m_\alpha=\dim\mfg_\alpha$ the multiplicity of $\alpha$. 
By choosing a lexicographic ordering of the roots we can define the set of positive roots $\Sigma^+$. Let $\Sigma_r^+$ be the set of reduced (also called short or indivisible) roots, that is, roots $\alpha\in \Sigma $ for which $\frac{1}{2}\alpha$ is not a root. 
Let $\Sigma_s^+\subset \Sigma^+$ be the set of simple roots. 
Note that $\Sigma_s^+ \subseteq \Sigma_r^+ \subseteq \Sigma^+ \subseteq \Sigma$; if $G$ is semisimple, then $\Sigma_s^+$ is a basis of $\mfa^*$.

Let $A$ and $N$ be the analytic subgroups of $G$ with the Lie algebras $ \mfa$ and $ \mfn$. 
On the Lie group level, we get the Iwasawa decomposition $G=KAN$, and the multiplication map $K \times A \times N \rightarrow G$ given by $(k, a, n)\mapsto kan$ is a diffeomorphism onto. 
Similarly, we also have the decomposition $G=NAK$. 
Let $H(g)$ denote the unique $\mfa$-component of $g\in G$ in the decomposition $g=k\exp(H(g))n$ with $k\in K, n\in N$;
and let $A(g)$ denote the unique $\mfa$-component in the decomposition $g=n\exp(A(g))k$.

The Killing form $B$ on $\mfg$ is positive definite on $\mfa$. For every $\lambda\in  \mfa^*$, there exists $H_\lambda\in  \mfa$ such that $\lambda(H)=\ip{ H_\lambda, H}$ for all $H\in \mfa$; this defines an inner product on $ \mfa^*$,
\[
\ip{\lambda, \mu}:=\ip{ H_\lambda, H_\mu },\quad \lambda, \mu \in \mfa^*.
\]
We denote in the sequel by $|\lambda|$ the corresponding norm of $\lambda\in \mfa^*$.

Let $\mfw$ be the Weyl group of the pair $(G, A)$. 
For any root $\alpha\in \Sigma$, the operator
\[
s_\alpha (\varphi)= \varphi-\frac{2\ip{\varphi, \alpha}}{\abso{\alpha}^2}\alpha, \ \varphi\in \mfa^*.
\]
is a reflection in the hyperplane $\sigma_\alpha=\set{\varphi\in \mfa^*|\ip{\varphi, \alpha}=0}$ and carries $\Sigma$ to itself.

Set 
\[\mfa_{\C}'=\set{H\in \mfa_{\C}\vert \alpha(H)\neq 0\ \text{for all} \ \alpha\in \Sigma}
\]
and $\mfa'=\mfa \cap \mfa_{\C}'$. 
We define the positive Weyl chamber of $\mfa$ associated with $\Sigma^+$ as 
\[
\mfa^+=\set{H\in \mfa \vert \alpha(H)> 0\ \text{for all} \ \alpha\in \Sigma^+}.
\]
We transfer this definition to $\mfa_{\C}^*$: 
\[
{\mfa_{\C}^{*}}'=\set{\mu\in \mfa_{\C}^*\vert \ip{\alpha, \mu}\neq 0\ \text{for all} \ \alpha\in \Sigma}, \quad {\mfa^{*}}'=\mfa^* \cap {\mfa_{\C}^{*}}'.
\]
Vectors in ${\mfa_{\C}^{*}}'$ are said to be regular.

We define the Weyl vector by
\[
\rho(H)=\frac{1}{2}\text{tr}(\ad H|_{\mfn}) =\frac{1}{2}\sum_{\alpha\in \Sigma^+}m_\alpha \alpha(H)
\]
for $H\in  \mfa$. 
Denote by $\Ca=\set{H\in \mfa: \alpha(H)\ge 0 \ \text{for all} \ \alpha\in \Sigma^+}$ the closure of $\mfa^+$ in $\mfa$. The polar decomposition of $G$ is given by $G=K\CA K$ where $\CA=\exp(\Ca)$, and we denote $x^+$ the $\CA$-component of $x\in G$ in this decomposition.
\subsection{Spherical functions} \label{spherical}

Harmonic analysis on symmetric spaces is built upon spherical functions.
It is important to note that our main references, Helgason \cite{helgason1962differential,helgason2022groups} and Gangolli and Varadarajan \cite{gangolli1988harmonic}, have different conventions in their notations. We chose to adopt Helgason's notations \cite{helgason2022groups} in our work; translation from \cite{gangolli1988harmonic} to \cite{helgason2022groups} is done by $\lambda \mapsto i\lambda$.

For each $\lambda\in \mfa_{\C}^*$, let $\varphi_\lambda(x)$ be the elementary spherical function given by 
\[
\varphi_\lambda(x)= \int_K e^{(i\lambda-\rho)(H(xk))} dk \quad \forall x\in G.
\]
Similarly, under the decomposition $G=NAK$ we set
\[
\varphi_\lambda(x)= \int_K e^{(i\lambda+\rho)(A(kx))} dk \quad \forall x\in G.
\]
These are exactly the joint eigenfunctions of all invariant differential operators on $G/K$.
We have the following functional equation: $\varphi_\lambda= \varphi_{s\lambda}$ for all $\lambda\in \mfa_{\C}^*, s\in \mfw$.

More information on spherical functions is given by the Harish-Chandra $\HCc$-function. An explicit expression for it as a meromorphic function on $\mfa_{\C}^*$ is given by the Gindikin-Karpelevich formula: For each $\lambda\in \mfa_{\C}^*$, 
\[
\HCc(\lambda) = c_0 \prod_{\alpha\in \Sigma_r^+}
\frac{ 2^{-i\ip{ \lambda, \alpha_0}} \Gamma\big(i\ip{ \lambda, \alpha_0}\big)}{
\Gamma\big(\frac{1}{2} (\frac{1}{2}m_\alpha +1+ i\ip{ \lambda, \alpha_0})\big)
\Gamma\big(\frac{1}{2}(\frac{1}{2}m_\alpha +m_{2\alpha}+ i\ip{ \lambda, \alpha_0})\big)},
\]
where $\alpha_0=\alpha/\ip{ \alpha, \alpha }$, $\Gamma$ is the classical $\Gamma$-function, and $c_0$ is a constant given by $\HCc(-i\rho)=1$. 
From the formula above, one can deduce for $\lambda \in \mfa^*$: 
\[
|\HCc(\lambda)|^{-2} \le C 
\begin{cases}
|\lambda|^{\nu-l} & \text{if}\ |\lambda|\le 1,\\
|\lambda|^{n-l} & \text{if}\ |\lambda|> 1.
\end{cases}
\]
where $n:=\dim G/K$, $l:=\dim A$, $d$ is denoted by the cardinality of $\Sigma_r^+$, and $\nu:=2d+l$ which is called the 'pseudo-dimension'. 
This implies a less precise, but shorter estimate
\begin{equation}\label{c-function}
\abso{\HCc(\lambda)}^{-1}\le 
C (1+|\lambda|)^{\frac{1}{2}(n-l)}.
\end{equation}
The other two properties of $\HCc$ we will use are that for $\lambda\in \mfa^*$ and $s\in \mfw$
\begin{equation}\label{c-conjugate}
|\HCc(s\lambda)|=|\HCc(\lambda)|,\quad \HCc(-\lambda)=\overline{\HCc(\lambda)}.
\end{equation}

Let now $\Sigma_s^+=\set{\alpha_1,\cdots, \alpha_r}$ be the simple system in $\Sigma^+$, and $\Lambda=\Z^+\alpha_1+\cdots+\Z^+ \alpha_r$ the set of all linear combinations $n_1\alpha_1 + \cdots + n_r\alpha_r$ $(n_i\in \Z^+)$. Set $\Lambda^+=\Lambda\setminus \{0\}$ and $\tilde\Lambda=\Z\Lambda=\Z\alpha_1+\cdots+\Z \alpha_r$. 
For any $\mu\in \tilde\Lambda$ and $s, t\in \mfw$ $(s\neq t)$ we define the following hyperplanes in $\mfa_{\C}^*$:
\[
\sigma_\mu=\set{\lambda\in\mfa_{\C}^*: \ip{\mu, \mu}=2i\ip{\mu,\lambda}},
\]
\[
\tau_\mu(s,t)=\set{\lambda\in \mfa_{\C}^*:i(s\lambda-t\lambda)=\mu}.
\]
Set $\sigma=\bigcup_{\mu\in \Lambda^+}\sigma_\mu$ and $\sigma^c=\mfa_{\C}^*\backslash \sigma$.
For $\mu=n_1\alpha_1 + \cdots + n_r\alpha_r \in \Lambda$, we denote by 
\[
m(\mu)=n_1 + \cdots + n_r
\]
the level of $\mu$.

If $\lambda\in \mfa_{\C}^*$ does not lie in any of the hyperplanes $\sigma_\mu$ or $\tau_\nu(s,t)$, then $\varphi_\lambda$ is decomposed into the Harish-Chandra series
\[
\varphi_\lambda(\exp H)= \sum_{s\in \mfw}\HCc(s \lambda)\e^{(is\lambda-\rho)(H)}\sum_{\mu\in \Lambda}\Gamma_\mu(s\lambda)\e^{-\mu(H)}, \quad  H\in  \mfa^+,
\]
where $\Gamma_\mu$ are coefficient functions on $\mfa^*$ determined by the recurrent relation
\[
\bracket{\ip{ \mu, \mu }-2i\ip{\mu, \lambda}}\Gamma_\mu(\lambda)= 2\sum_{\alpha\in \Sigma^+}m_\alpha\sum_{\substack{ k\ge 1,\\ \mu-2k\alpha\in \Lambda}}\Gamma_{\mu-2k\alpha}(\lambda)\bracket{\ip{ \mu+\rho-2k\alpha,\alpha}-i\ip{ \alpha, \lambda}},
\]
with initial function $\Gamma_0\equiv 1$. 

Let $C(G//K)$ denote the space of continuous functions on $G$ which are bi-invariant under $K$. 
Let $C_c(G//K)$ be the set of all functions in $C(G//K)$ with compact support. 
The spherical transform, called also the Harish-Chandra transform, of $f\in C_c(G//K)$ is defined by
\[
(\mH f)(\lambda)=\int_G f(x)\varphi_{-\lambda}(x)d x, \quad \lambda\in  \mfa_{\C}^*.
\]
The Harish-Chandra inversion formula is given by
\[
f(x)=C\int_{ \mfa^*}(\mH f)(\lambda)\varphi_\lambda(x) \abso{\HCc (\lambda)}^{-2}d\lambda, \quad x\in G
\]
where $C$ is a constant associated with $G$. 

\subsection{Symmetric space}\label{symmetric space}
When $G$ is a semisimple, connected Lie group with finite center, the homogeneous space $G/K$ can be identified with the solvable Lie group $S=AN$ with the Lie algebra $\mfa\oplus \mfn$. 
 The multiplication and inverse for $(a,x),\, (b,y)\in S$ are given by
 \[
 (a,x)\,(b,y)=(ab,bxb^{-1}y) 
 \]
 and 
 \[
 (a,x)^{-1}=(a^{-1}, ax^{-1} a^{-1}), 
 \]
which is well defined as $\mfn$ is stable under $\ad(\mfa)$.

The group $G$ is unimodular, while $S$ is not (unless $G=A$) and has the modular function given by 
\[
\delta(s)=\e^{-2\rho(\log a)}, \ \ s= (a,x)\in S.
\]
It admits an extension to $G$ defined by $\delta(g):=\e^{-2\rho(H(g))}(g\in G)$.
The left and right Haar measures of $S$ are given by
\[
d_ls=d a d n,
\]
\[
d_rs=\delta^{-1}(s) d_ls=\e^{2\rho(\log a)}d a d n.
\]

\subsection{Distinguished Laplacian and Laplace-Beltrami operator}\label{La-LB}
Keep the assumptions of Section \ref{symmetric space}.
The Killing form $B$ of $G$ is positive definite on $\mfp$. This defines a Riemannian structure on $G/K$ and the Laplace--Beltrami operator $\Delta$. It has also a coordinate description as described below.

The bilinear form $B_\theta: (X,Y)\mapsto -B(X,\theta Y)$ is positive definite on $\mfg$ and the Iwasawa decomposition $\mfg=\mfk\oplus\mfa\oplus\mfn$ is orthogonal with respect to it. We choose and orthonormal basis $(H_1,\dots,H_l,X_1,\dots,X_{\dim \mfg})$ in $\mfg$ so that $(H_1,\dots,H_l)$ is a basis of $\mfa$ and $(X_1,\dots,X_{n-l})$ a basis of $\mfn$. With these viewed as left-invariant vector fields on $\mfg$, we have then \cite{bougerol1983exemples}
$$
\Delta = \sum_{i=1}^{l} H_i^2+2\sum_{j=1}^{n-l} X_j^2+\sum_{i=1}^l H_i \cdot (H_i\delta)(e).
$$
For a function $f$ on $G$, set $\tau (f) = \check f$ and $\check f(s):=f(s^{-1})$. Setting $-\widetilde X = \tau \circ X \circ \tau$ for $X\in\mfg$, we get a right-invariant vector field; the operator $L$ defined by
$$
-L = \sum_{i=1}^{l}\widetilde H_i^2+2\sum_{j=1}^{n-l} \widetilde X_j^2
$$
is a distinguished Laplace operator on $S=G/K$ and 
has a special relationship with $\Delta$\cite{cowling1994spectral} :
\[
\delta^{-1/2}\circ\tau L \tau\circ\delta^{1/2} = -\Delta - \abso{\rho}^2=: \Delta_\rho.
\]
The shifted operator $\Delta_\rho$ has the spectrum $[0, \infty)$. Both $\Delta_\rho$ and $L$ extend to positive and self-adjoint operators on $L^2(S, d_l)$, with the same notations for their extensions.

Let $F$ be a bounded Borel function on $[0, \infty)$. 
We can define bounded operators $F(L)$ and $F(\Delta_\rho)$ on $L^2(S)$ via Borel functional calculus.  
Note that $L$ is right-invariant while $\Delta_\rho$ is left-invariant, we can define $k_\LB$ and $k_\La$ to be the kernels of $F(\Delta_\rho)$ and $F(L)$ respectively, a priori distributions, such that 
\begin{equation}\label{kernel of functions}
    \begin{split}
        F(\Delta_\rho)f=f *k_\LB \quad  
        \text{and}\quad  
        F(L)f=k_\La*f, \quad f\in L^2(S).
    \end{split}
\end{equation}
The connection between $L$ and $\Delta$ implies that
\begin{equation}\label{relation of kernels}
	k_\LB=\delta^{1/2} k_\La,
\end{equation}
which was pointed out by Giulini and Mauceri \cite{giulini1993analysis} (also see \cite{bougerol1983exemples}).
It is known that 
$k_\LB$ is $K$-biinvariant. Thus, not only $k_\LB$ but $k_\La$ can be regarded as a function on $G$ by the formula above.
By inverse formula for the spherical transform, we obtain the exact formula of the kernel as follows:
\[
k_\LB(x)=C\int_{ \mfa^*}F(|\lambda|^2)\varphi_\lambda(x) \abso{\HCc (\lambda)}^{-2}d\lambda, \quad x\in G.
\]

\section{Asymptotic behavior of $\varphi_\lambda$}\label{Asymptotic behavior}
In this section, we shall introduce some notations and properties about the asymptotics of the elementary spherical functions. The main result is by Harish-Chandra \cite{HARISHCHANDRA1975104}
, but we are systematically citing the book of Gangolli and Varadarajan \cite{gangolli1988harmonic}.

In this section, $G$ is a group in class $\mathcal{H}$. We use all the notations introduced above.

Fix $H_0\in \Ca$ which is not in the center of $\mfg$. This is equivalent to say that the subset $F\subset \Sigma_s^+$ of simple roots vanishing at $H_0$ is not equal to $\Sigma_s^+$. Set $\mfa_F = \{ H\in \mfa: \a(H)=0\ \text{for all}\ \a\in F\}$, and let $\mfm_{1F}$ be the centralizer of $\mfa_F$ in $\mfg$. If we denote by $M_{1F}$ the centralizer of $\mfa_F$ in $G$, then $\mfm_{1F}$ is corresponding Lie algebra of $M_{1F}$.

Denote by $\Sigma_F^+ \subset \Sigma^+$ the set of positive roots vanishing at $H_0$ and $\Sigma_F:=\Sigma_F^+\cup -\Sigma_F^+$. Then $F$ and $\Sigma_F^+$ are the simple and positive systems of the pair $(\mfm_{1F}, \mfa)$ respectively.
With $\mfm=Z_{\mfk}(\mfa)$, we have the restricted root space decomposition
\[
\mfm_{1F}=\mfm\oplus\mfa\oplus \bigoplus_{\lambda\in \Sigma_F}\mfg_\lambda,
\]
where as expected 
\[
\mfg_\lambda = \set{X\in \mfm_{1F}: \Bracket{H,X}=\lambda(H)X\ \text{for all}\ H\in \mfa}
\] 
for each $\lambda\in \Sigma_F$. 

Let $\mfk_{F}=\mfm_{1F}\cap \mfk$ 
and $\mfn_{1F}=\mfm_{1F}\cap \mfn=\bigoplus_{\lambda\in \Sigma_F^+}\mfg_\lambda$. 
Then we have 
the Iwasawa decomposition $\mfm_{1F}=\mfk_F\oplus\mfa\oplus\mfn_{1F}$. 
Let $K_F=K\cap M_{1F}$ and $N_{1F}=\exp \mfn_{1F}$, on Lie group level, we also have the decomposition 
$M_{1F}=K_F A N_{1F}$.
We define $\theta_F=\theta|_{M_{1F}}$ and $B_F=B|_{\bracket{\mfm_{1F}}_{\C}\times\bracket{\mfm_{1F}}_{\C}}$, then $\theta_F$ is a Cartan involution of $M_{1F}$. It is important that $M_{1F}$ is also of class $\mathcal{H}$ and $(M_{1F},K_F,\theta_F,B_F)$ inherit all the properties of $(G,K,\theta,B)$\cite[Theorem 2.3.2]{gangolli1988harmonic}.

We denote the Weyl group of pair $(\mfm_{1F},\mfa)$ by $\mfw_F$ and define $\tau_F(H)=\min_{\alpha\in \Sigma^+\backslash \Sigma_F^+}\abso{\alpha(H)}$ for all $H\in \mfa$. 
Let $\HCc_F$ denote the Harish-Chandra $\HCc$-function with respect to $M_{1F}$. 
The elementary spherical functions on $M_{1F}$ corresponding to $\lambda\in \mfa_{\C}^*$ are given by
\[
\theta_\lambda(m)=\int_{K_F}\e^{(i\lambda-\rho_F)(H(mk_F))}dk_F,
\]
where $\rho_F:=\frac{1}{2}\sum_{\alpha\in \Sigma_F^+}m_\alpha \alpha$. 

An important part of our proofs relies on Theorem \ref{thm593} below, which is \cite[Theorem 5.9.3 (b)]{gangolli1988harmonic}. 
To state it, we need the following result. It is a proposition in \cite{gangolli1988harmonic} since the functions $\psi_\lambda$ on $M_{1F}$ are defined there by another formula \cite[(5.8.23)]{gangolli1988harmonic}; we take it however for a definition.

\begin{notation}[\text{\cite[Proposition 5.8.6]{gangolli1988harmonic}}]\label{thm586}
	Let $\lambda\in \mfa^*$ be regular. Then 
	\[
	\psi_\lambda(m)=\abso{\mfw_F}^{-1} \sum_{s\in \mfw} (\HCc(s\lambda)/\HCc_F(s\lambda))\theta_{s\lambda}(m)
	\]
	for all $m\in M_{1F}$.
\end{notation}

Next, we introduce the following theorem about the asymptotic behavior of $\varphi_\lambda$ in a region including a wall of the positive chamber. 
For any $\zeta> 0$, define 
\begin{equation}\label{conical-region}
	A^+(H_0:\zeta):=\set{a\in\CA : \tau_F(\log a) > \zeta|\log a| }.
\end{equation}
This is a conical open set in $\CA$, and it contains $\exp H_0$ when $\zeta$ is small enough.

\begin{theorem}[\text{\cite[Theorem 5.9.3 (b)]{gangolli1988harmonic}}]\label{thm593}
Fix $\zeta >0$. Then we can find constants $\kappa=\kappa(\zeta)>0$, 
$C=C(\zeta)>0$, and $\iota=\iota(\zeta)>0$ such that for all $\lambda\in \mfa^*$ and all $a\in A^+(H_0: \zeta)$,
\[
\big|\e^{\rho(\log a)}\varphi_\lambda(a)-\e^{\rho_F(\log a)}\psi_\lambda(a)\big| 
\le C(1+\abso{\lambda})^\iota\e^{-\kappa \abso{\log a}}.
\]
\end{theorem}

\subsection{Estimates of the constants in Theorem \ref{thm593}}

We are interested in explicit values of the constants in the theorem above, and especially of $\iota$. Theorem \ref{thm593} as stated is in fact a particular case of \cite[Theorem 5.9.3 (b)]{gangolli1988harmonic}: the latter contains estimates for a differential operator $b$ acting on $\varphi_\lambda$, and we apply it in Section \ref{sec-main} with $b=\id$. We would like to estimate the constant $\iota$, which is denoted by $s$ in the book \cite{gangolli1988harmonic}, in the case of a general operator $b$, since the proof is almost the same.

Extracting constant from the proofs of \cite[Theorem 5.9.3 (b)]{gangolli1988harmonic} and its preceding statements is a long task requiring a lot of citations. 
It is impossible to explain every symbol appearing in the discussion and to keep a reasonable volume, so we invite the reader to consult the book \cite{gangolli1988harmonic} in parallel. Triple-numbered citations like Proposition 5.8.3 or (5.9.2) will be always from this book.  Any new notations introduced in this subsection will be used only within it.

First to mention are 
 the notations concerning parabolic subgroups. Those introduced above suit our proofs in Lemma \ref{main lemma}, but are different from \cite{gangolli1988harmonic}. In this subsection we denote, following the book \cite{gangolli1988harmonic}: $M_{10}=M_{1F}$, $\mfm_{10}=\mfm_{1F}$, $\mfw_0=\mfw_F$, $\tau_0=\tau_F$ and $\rho_0=\rho_F$.

Let $\mfs_{10}$ denote the algebra $\mfm_{10} \cap \mfp$ (we note that in \cite{gangolli1988harmonic}, the complement of $\mfk$ in the Cartan decomposition is denoted by $\mfs$ rather than $\mfp$).
Denote by $U(\mfg)$ and $U(\mfs_{10})$ the universal enveloping algebras of $\mfg$ and $\mfs_{10}$ respectively. There is a projection $\beta_0$ of $U(\mfg)$ onto $U(\mfs_{10})$ defined in (5.3.32) and (5.3.33).
Next, a map $\gamma_0$ is defined in (5.9.2) by
$$
\gamma_0(b) = d_{P_0} \circ \beta_0(b) \circ d_{P_0}^{-1},\quad b\in U(\mfg)
$$
where $d_{P_0}$ is a homomorphism from $M_{10}$ to $\R_+$ defined on p. 208 referring to the formula (2.4.8). For $a\in A$, we have $d_{P_0}(a) = e^{(\rho-\rho_0)(\log a)}$ by Lemma 5.6.1. 

The space $\mfa^*_{\C}$ is denoted by ${\mathcal F}$ in \cite[\S 3.1]{gangolli1988harmonic}, and is decomposed as $\mathcal{F} = \mathcal{F}_R\oplus \mathcal{F}_I$ with $\mathcal{F}_R=\mfa^*$ and $\mathcal{F}_I = i\mfa^*$. For $\lambda\in\mathcal{F}$ this gives $\lambda=\lambda_{\bf R}+\lambda_{\bf I}$.
For any $\kappa>0$, the set ${\mathcal F}_I(\kappa)$ is defined in Lemma 4.3.3 as
$$
{\mathcal F}_I(\kappa) = \{ \lambda\in\mathcal{F}: |\lambda_{\bf R}| <\kappa \}.
$$
In particular, it contains $\mathcal{F}_I$. 

It is important to recall again that a spherical function denoted by $\varphi_\lambda$ in \cite{helgason1979differential} would be denoted $\varphi_{i\lambda}$ in \cite{gangolli1988harmonic}, so that $\varphi_\lambda$ of \cite{gangolli1988harmonic} is positive definite when $\lambda\in {\mathcal F}_I$ (Proposition 3.1.4), and the integral in the inversion formula of the spherical transform is over ${\mathcal F}_I$ as well (Theorem 6.4.1). Thus, in Proposition \ref{choice-of-s} below we consider the case $\lambda\in {\mathcal F}_I$ which implies $\lambda_{\bf R}=0$; in Corollary \ref{choice-of-iota}, we return to the notations of the rest of the paper.

We can now state the theorem:

\medskip
\noindent {\bf Theorem 3.2$'$} \cite[Theorem 5.9.3 (b)]{gangolli1988harmonic}\label{thm593'}
{\it Fix $\zeta >0$. Then we can find constants $\kappa=\kappa(\zeta)>0$, and, for each $b\in U(\mfg)$, constants 
$C=C(\zeta,b)>0$, $s=s(\zeta,b)>0$ such that for all $\lambda\in {\mathcal F}_I(\kappa)$ and all $a\in A^+(H_0: \zeta)$,
\begin{equation}\label{593-ineq}
\big|\e^{\rho(\log a)}(b\varphi_\lambda)(a)-\e^{\rho_0(\log a)} (\gamma_0(b) \psi_\lambda)(a)\big| 
\le C(1+\abso{\lambda})^s\e^{-\kappa \abso{\log a}}.
\end{equation}
}

Our result is now:

\begin{proposition}\label{choice-of-s}
For regular $\lambda\in {\mathcal F}_I$, the inequality \eqref{593-ineq} holds with the constant $s = 6d +1 +\deg b$ where $d=\abso{\Sigma_r^+}$ is the number of indivisible roots.
\end{proposition}
\begin{proof}
As we will see, this constant is determined by the degrees of differential operators involved, so we will follow the proofs very briefly while keeping a particular attention at this point.

Two more remarks on notations are also needed. If $\mu$ is a differential operator and $f$ a function (on $G$, or its subgroup), then we denote its value at a point $x$ by $(\mu f)(x)$, but in \cite{gangolli1988harmonic} it is denoted also by $f(x;\mu)$. Also, the spherical functions in \cite{gangolli1988harmonic} are denoted both $\varphi_\lambda(x)$ and $\varphi(\lambda:x)$.

The proof of Theorem 5.9.3(b) relies on Proposition 5.9.2(b) with the same power $s$.
In its proof in turn, estimates with $(1+|\lambda|)^s$ come from two sides:
{ \renewcommand{\theenumi}{(\alph{enumi})}
\begin{enumerate}
    \item Formula (5.8.8) of Proposition 5.8.3 estimating $|\Phi_0(\lambda:m\exp H;\mu) - \Theta (\lambda:m\exp H;\mu)|$, with $\mu=\gamma_0(b)$.
    The function $\Phi_0$ is defined in (5.4.12), and $\Theta$ in Proposition 5.8.1. Formula (5.8.8) is a vector-valued inequality: by (5.4.12), for $\lambda\in \mfa^*_{\C}$, $m\in M_{10}$
    $$
    \Phi_0(\lambda:m) = \begin{pmatrix} (v_1\circ d_{P_0}) \varphi_\lambda(m) \\
    \vdots\\ 
    (v_k\circ d_{P_0}) \varphi_\lambda(m)\end{pmatrix}
    $$
    with operators $v_1, \dots, v_k\in U(\mfs_{10})$ defined in (5.4.9). In Proposition 5.9.2, only the first coordinate is used, in which $v_1=1$ by (5.4.5).
\item The estimates of $(\mu_i \varphi_\lambda)(m \exp H)$ with $m\in M_{10}$, $H\in\mfa$, which are obtained by Lemma 5.6.3 with $\mu=d_{P_0} \circ \mu_i \circ d_{P_0}^{-1}$; here $\mu_i$, $1\le i\le N$, are operators in $U(\mfm_{10})$ whose existence follows from Proposition 5.3.10 with $g=b$ (since $\varphi_\lambda$ is $K$-biinvariant, we have $(b\,\varphi_\lambda)(m)=(\delta'(b)\varphi_\lambda)(m)$ by (5.3.26), at $m\in M^+_{10}A_+$ as in the proof of Proposition 5.9.2). We get also from Proposition 5.3.10 that $\deg \mu_i\le \deg b$ for all $i$.
\end{enumerate}
}
{\bf Discussion of the case (b).}
Case (a) will require a longer discussion, but it is eventually reduced to the same Lemma 5.6.3, so we continue first with the case (b).
Let us state the lemma in question
\cite[Lemma 5.6.3]{gangolli1988harmonic}: {\it Denote by $U(\mfm_{10})$ the universal enveloping algebra of $\mfm_{10}$.
    Fix $\mu\in U(\mfm_{10})$. Then there are constants $C=C(\mu)>0, s=s(\mu)\ge 0$ such that
    \begin{equation}\label{phi-mu-dP0}
    | (\mu\circ d_{P_0}) \varphi_\lambda(m) | \le C e^{|\lambda_{\bf R}| \sigma(m) } (1+|\lambda|)^s (1+\sigma(m))^s \Xi_0(m)
    \end{equation}
    for all $\lambda\in \mfa_{\C}^*, m\in M_{10}$.
}
The function $\sigma$ is defined in (4.6.24), but in fact we do not need to consider it since this factor does not depend on $\lambda$. We recall again that notations of \cite{gangolli1988harmonic} are different from that of Helgason and of those used in our paper, so that for $\lambda\in \mfa^*\subset \mfa^*_{\C}$ we have $\lambda_{\bf R}=0$ in the cited Lemma.

In the proof of Lemma 5.6.3, the constant $s$ comes from the formula (4.6.5) in Proposition 4.6.2:
$$
|\varphi(\lambda; \partial(v):x; u)| \le C (1+|\lambda|)^{\deg u} \big(1+h(x)\big)^{\deg v} \varphi(\lambda_{\bf R}:x).
$$
When applied in Lemma 5.6.3, one puts $v=1$ so that $\partial(v)=\id$, then $x=m$ and $u=d_{P_0}^{-1}\circ\mu\circ d_{P_0}$, so that $\deg u = \deg \mu$. The function $h$ is defined in the same Proposition 4.6.2 but it is not significant in our case since $\deg v=0$.

We see now that the constant $s$ appearing in Lemma 5.6.3 is $\deg \mu$, and by our previous considerations, it is bounded by $\deg b$ in the case (b).

{\bf Discussion of the case (a).}
In the proof of Proposition 5.8.3, the estimate (5.8.8) follows from (5.8.16), where one can replace ${\exp\big(\Gamma_0(\lambda:H)\big) \Theta^\mu(\lambda:m)}$ by $\Theta(\lambda:m\exp H;\mu)$ according to (5.8.18) and (5.8.20).
And (5.8.16) is an application of Proposition 5.7.4, precisely (5.7.26). One should note that (5.7.26) is an inequality involving functions of ${\bf t}\in \R^q$, whereas in (5.8.16) this variable is hidden: we suppose that $H=\sum_{i=1}^q t_i H_i$, where $H_1,\dots,H_q$ are chosen on p.224.

The factor $(1+|\lambda|)^s$ in (5.8.16) is a product of two others. The first one comes from the factor $(1+\|{\bf\Gamma}\|)^{2(k-1)}$ in (5.7.26) since $\Gamma_i=\Gamma_0(\lambda:H_i)$, with $\Gamma_0$ defined in (5.5.18). For every $i$, the norm of the matrix $\|\Gamma_0(\lambda:H_i)\|$ is bounded by $|\lambda|\,|H_i|$ since its eigenvalues are of the form $H_i(s^{-1}\lambda)$, $s\in \mfw$, by Theorem 5.5.13. Moreover, $k$ in this formula is the dimension of the space of values of the function $\bf f$; in (5.8.16), the function under study is $\Phi_0$ (with $\lambda, m, \mu$ fixed), and its values are $k$-dimensional with $k$ as in (5.4.12), that is, $k=[\mfw:\mfw_0]$, by (5.4.4).

The second factor contributing to $(1+|\lambda|)^s$ in (5.8.16) is $\|G_{i,\lambda,m}\|_{\beta |\lambda_{\bf R}| ,s}$ estimated in (5.8.9), with the function $G_{i,\lambda,m}$ defined in (5.6.4), and the power $s$ is given by Lemma 5.6.6; recall that this lemma is applied with $\mu=\gamma_0(b)$. Note also that (5.8.9) does not add an exponential factor in $\lambda$ since $\lambda_{\bf R}=0$ for $\lambda\in \mfa^*$.
    
Altogether, we get $(1+|\lambda|)$ in the power $2([\mfw:\mfw_0]-1)+s$, where $s$ is given by Lemma 5.6.6 (second inequality) with $\mu=\gamma_0(b)$.
The index $[\mfw:\mfw_0]$ can however be large and is not satisfactory for us. But actually it replaced by a better estimate, and let us see how.

{\bf Improving the estimate $(1+\|\Gamma\|)^{2(k-1)}$.}
We have seen that this factor comes from (5.7.26), with $\Gamma_i = \Gamma_0(\lambda:H_i)$, $i=1,\dots,q$. Formula (5.7.26) should read
$$
|{\bf f(t)} - \exp(t_1 \Gamma_1+\dots+t_q \Gamma_q) \theta| \le C (1+\|\Gamma\|)^{2(k-1)} (1+|{\bf t}|)^{s+2k-1} 
\max_i \| {\bf g} _i\|_{b,s} e^{-(a\eta/3)|{\bf t}|}
$$
(there is a misprint in the book where the exponent is $\exp(-t_1 \Gamma_1+\dots+t_q \Gamma_q)$; the first minus sign is unnecessary, as can be figured out from the proof, or compared with (5.7.14); when applied in Proposition 5.8.3, this minus is not supposed either).
This inequality is obtained by estimating first $|{\bf f^*}({\bf t})-\theta|$ with ${\bf f^*}({\bf t}) = \exp(-t_1 \Gamma_1-\dots-t_q \Gamma_q) {\bf f}({\bf t})$ (notations of (5.7.15)) then multiplying by $\exp(t_1 \Gamma_1+\dots+t_q \Gamma_q)$. In our case, 
$$
t_1 \Gamma_1+\dots+t_q \Gamma_q = t_1 \Gamma_0(\lambda:H_1)+\dots+t_q \Gamma_0(\lambda:H_q) = \Gamma_0(\lambda:H),
$$
since $\Gamma_0$ is linear in $H_i$ (see (5.5.16)-(5.5.19)). Let us write $L=\Gamma_0(\lambda:H)$ till the end of this proof. Thus, we are estimating
$$
|{\bf f(t)} - \exp(L) \theta| \le \|\exp(L)\| \,|{\bf f^*(t)}-\theta|.
$$
The estimate for $|{\bf f^*(t)}-\theta|$ is obtained as $|{\bf f^*(t)}-{\bf f^*}(t,\dots,t)| + |{\bf f^*}(t,\dots,t)-\theta|$ where $t=\min({\bf t})$, and both are eventually reduced to Corollary 5.7.2 of Lemma 5.7.1. We would like to replace it by the following:

{\bf Claim.} {\it If $\lambda\in \mfa^*$ is regular, then $\|\exp(L)\| \le C_k (1+|\lambda|)^{2\xi}$ for every $H
\in \mfa$, where $\xi = \max_{1\le j\le k} {\deg u_j}$ and $u_j$ are defined in (5.5.14).}
\begin{proof}[Proof of Claim]
    Let $\bar\mfw$ denote the coset space $\mfw/\mfw_0$. By Theorem 5.5.13, $L=\Gamma_0(\lambda:H)$ is diagonalizable: for $\bar s\in \bar\mfw$, the vector $f_{\bar s}(\lambda)$ of (5.5.34) is its eigenvector with the eigenvalue $H(s^{-1}\lambda)$, and these vectors form a basis of $\C^k$ (recall that $k=|\bar\mfw|$). The matrix $F(\lambda)$ with columns $f_{\bar s}(\lambda)$ is the transition matrix to the Jordan form $J(\lambda)$ of $L=F(\lambda) J(\lambda) F(\lambda)^{-1}$, so that
    $\exp L=F(\lambda) \exp\big( J(\lambda) \big) F(\lambda)^{-1}.$ We have now
    $$
    \| \exp L\| \le \| F(\lambda) \|\, \|\exp\big( J(\lambda) \big) \| \, \|F(\lambda)^{-1} \|
    \le \|F(\lambda)\| \,\| F(\lambda)^{-1}\|
    $$
    since $J(\lambda)$ is diagonal with purely imaginary values on the diagonal (recall that we assume $\lambda\in 
    i \mfa^*$).
    
    The norm of $F(\lambda)$ is up to a constant (depending on $k$) equal to the maximum of its coordinates, that is, of $u_j(s^{-1}\lambda)$ over $1\le j\le k$ and $s\in\mfw$, by Theorem 5.5.13. Every $u_j$ is a polynomial function on $\mfa^*$ of degree $\deg u_j$, chosen in (5.5.14). We can thus estimate $\|F(\lambda)\| \le C_k \max_j (1+|\lambda|)^{\deg u_j}$.
    
    The inverse matrix has, by Lemma 5.5.10, coordinates $|\mfw_0| \,u^i(\bar s^{-1} \lambda)$ where $(u^i)$ is a basis dual to $(u_j)$ with respect to the bilinear form (5.5.25). Both bases span the same linear space. Belonging to the linear span of $(u_j)$, every $u^i$ has degree at most $\max_j {\deg u_j}$, and this allows to estimate $\|F(\lambda)^{-1}\|$ by $C_k \max_j (1+|\lambda|)^{\deg u_j}$ and finally
    $\|\exp L\| \le C_k \max_j (1+|\lambda|)^{2\deg u_j}$, what proves the claim.
\end{proof}

Return now to the proof of (5.7.26). We can use now the {\bf Claim} above instead of every reference to Lemma 5.7.1 or Corollary 5.7.2. One occurrence is (5.7.27), but let us first recall the notations (5.7.7)-(5.7.8) which imply that for any continuous $\C^k$-valued function ${\bf h}$ on $\R^q_+$ and any ${\bf t}\in \R^q_+$,
$$
|{\bf h(t)}| \le k^{1/2} e^{b|{\bf t}| - a\min ({\bf t}) } (1+|{\bf t}|)^s \|{\bf h}\|_{b,s}.
$$
Note that $|{\bf t}| = |t_1|+\dots+|t_q|$ and $\min({\bf t})=\min (t_1,\dots, t_q)$ by (5.6.6). In particular,
\[
|{\bf h}((t,\dots, t))| \le k^{1/2} e^{bqt-at} (1+ qt)^s \|{\bf h}\|_{b,s}.
\]
This together with the Claim implies that (5.7.27) can be improved to
$$
\sum_{1\le i\le q} |{\bf g}_i^*(t,\dots,t)| \le C  (1+qt)^s (1+\|\lambda\|)^{2\xi}  e^{qbt-at}\max_{1\le i\le q}\|{\bf g}_i\|_{b,s}.
$$
By (5.7.23), $qb-a<a\eta/3-a<-2a/3$. The existence of the limit $\theta$ follows as in the original proof, and this limit is given by (5.7.12). Up to a change of notations, this is also (5.7.20). We can now estimate
$$
|{\bf f^*}(t,\dots,t)-\theta| \le C \max_{1\le i\le q} \|{\bf g}_i\|_{b,s} (1+|\lambda|)^{2\xi} \int_t^\infty (1+qu)^s e^{(qb-a)u} du,
$$
with the integral
$$
\int_0^\infty (1+qt+qv)^s e^{(qb-a)(t+v)} dv \le C_{a,b,s} (1+t)^s e^{(qb-a)t}.
$$
Next comes the estimate of $|{\bf f^*}({\bf t})-{\bf f^*}(t,\dots,t)|$, done similarly to the calculations between (5.7.21) and (5.7.22). We arrive at
$$
|{\bf f^*}({\bf t})-{\bf f^*}(t,\dots,t)| \le C \max_{1\le i\le q} \|{\bf g}_i\|_{b,s} (1+|\lambda|)^{2\xi} (1+|{\bf t}|)^s e^{b|{\bf t}|-a \min({\bf t})}
$$
instead of (5.7.22), and conclude that for ${\bf t}\in S_\eta$ (this set is defined in (5.7.4), and within it $t\ge\eta|{\bf t}|$)
$$
|{\bf f}({\bf t}) - \exp(L) \theta| \le (1+|\lambda|)^{4\xi} \max_{1\le i\le q} \|{\bf g}_i\|_{b,s} (1+t)^s e^{(b-\eta)|{\bf t}|},
$$
allowing to replace $(1+\|\Gamma\|)^{2(k-1)}$ by $(1+|\lambda|)^{4\xi}$ in (5.7.26).

{\bf Collecting the powers of $(1+|\lambda|)$.} The final power of $(1+|\lambda|)$ is the maximum between two cases: in case (b), $s\le \deg b$; in case (a), $s$ is bounded by the sum of $4\xi = 4\max_j \deg u_j$ with $u_j$ of (5.5.14 ), and of the power $s$ given by Lemma 5.6.6 (second inequality) with $\mu=\gamma_0(b)$, of degree $\deg\mu\le\deg b$. We will proceed now to estimate this last power.

Lemma 5.6.6 with its two inequalities follows directly from Lemma 5.6.4, first and second inequalities respectively. The power $s$ and the operator $\mu$ are the same in both. And the second inequality of Lemma 5.6.4 is reduced to the first one, with another $\mu$; we need to estimate its degree. For a given $i$, the operator $E_i=E_{H_i}$ is defined in (5.4.14) as a $k\times k$ matrix with nonzero entries(only in the first column) of the type
$$
\sum_{p=1}^{p_i} \psi_{H_i,jp} \mu_{H_i,jp},
$$
$j=1,\dots,k$, where $\psi_{H_i,jp}$ are functions in $\cal R^+$ (that is, smooth functions on $A^+$, see a remark after formula (5.2.7) and notations after (4.1.26)), and $\mu_{H_i,jp}$ are differential operators in $U(\mfm_{10})$ of degree
$$
\deg \mu_{H_i:jp} \le \deg H_i + \deg v_j = 1 + \deg v_j,
$$
by (5.4.13).
The operators $v_j$ are introduced by the formula (5.4.9), and we will yet return to them. The product $\mu E_i$ is a matrix again with entries(only in the first column) 
\begin{equation}\label{entries of matrix}
    \sum_{p=1}^{p_i} \mu\circ (\psi_{H_i,jp} \mu_{H_i,jp}) = \sum_{p=1}^{p_i}\sum_{l=1}^{l_i} \psi_{ijpl}' \mu'_{ijpl} \mu_{H_i,jp},
\end{equation}
$j=1,\dots, k$, with some $\psi_{ijpl}'\in \cal R^+$, $\mu'_{ijpl}\in U(\mfm_{10})$, and each entry in (\ref{entries of matrix}) has degree $\le \deg\mu+1+\max_j \deg v_j$. This is thus the maximal degree of $\mu$ appearing in the first inequality of Lemma 5.6.4. The last reduction to make is to use the formula (5.4.12) defining $\Phi_0$, also cited above; estimates for $\Phi_0(\lambda:m;\mu)$ follow from those for $\varphi(\lambda:m; \mu\circ d_{P_0})$ in Lemma 5.6.3, but with $\mu\circ v_j$, $j=1,\dots,k$, instead of $\mu$. This increases the bound for the degree in (\ref{entries of matrix}) to
\begin{equation}\label{deg_mu}
\deg \mu+1+2\max_j \deg v_j
\end{equation}
with $(v_j)$ of (5.4.9). Now, as we have seen in the discussion of the case (b) above, the constant $s$ in Lemma 5.6.3 is $\deg \mu$, so that $s= s(\mu) \le \deg b+1+2\max_j \deg v_j$ with $\mu=\gamma_0(b)$.

{\bf The total estimate} for $s$ in the initial Theorem 5.9.3 is now
\begin{equation}\label{s-total}
    s \le \deg b+1+2\max_j \deg v_j + 4 \max_j \deg u_j
\end{equation}
with $(u_j)$ of (5.5.14) and $(v_j)$ of (5.4.9).

Concentrate now on the degrees of $v_j$. These operators are introduced by the formula (5.4.9): $v_j\in \mathfrak{Q}_0=U(\mathfrak{s}_{10})^{K_0}$ are such that $\gamma_{\mfm_{10}/\mfa}(v_j)=u_j$; the elements $u_j\in U(\mfa)$ are homogeneous and $\mfw_0$-invariant. The homomorphism $\gamma_{\mfm_{10}/\mfa}$ is defined by the formula (2.6.53) and Theorem 2.6.7. By the proof of Theorem 2.6.7 (p.~93-94) $v_j$ can be chosen of the same degree as $u_j$.

Now by (5.4.1-5.4.5) $u_j$ are in $\mathfrak H$, a graded subspace of $U(\mfa)$ such that $U(\mfa)$ is the direct sum of $\mathfrak H$ and $U(\mfa)I_{\mfw}^+$, where $I_{\mfw}$ is the space of $\mfw$-invariant polynomials in $U(\mfa)$ and $I_{\mfw}^+$ is the subset of polynomials in $I_{\mfw}$ of positive degree. This situation is described by general theory in \cite[Appendix 4.15]{Varadarajan1974Lie}. In particular, \cite[Theorem 4.15.28]{Varadarajan1974Lie} states that the multiplication map $I_{\mfw}\otimes \mathfrak{H} \to U(\mfa)$ is an isomorphism, and the Poincar\'e series of $\mathfrak H$ is given by
\begin{equation}\label{PHt}
P_{\mathfrak H}(t) = \prod_{1\le i\le l} (1+t+\dots +t^{d_i-1}),
\end{equation}
where $d_i=\deg p_i$ are the degrees of polynomials generating $I_{\mfw}$. There are exactly $l$ of them, see \cite[Theorem 4.9.3]{Varadarajan1974Lie}; moreover, by (4.15.34) we know that $\sum_{i=1}^l (d_i-1)$ is the number of reflections in $\mfw$. This is exactly the number of indivisible roots in $\Delta^+$ (see for example Corollary 4.15.16, and note that $s_\alpha = s_{2\alpha}$ if both are roots). This gives us an estimate
$$
\sum_{i=1}^l (d_i-1) \le d.
$$
By \eqref{PHt}, the left hand side is also the highest degree of polynomials in $\mathfrak H$, whence we conclude that $\deg v_j\le d$.

The reasoning for $u_j$ is similar, just done in a more abstract setting; the bound is $d$ as well. 

We obtain finally:
$$
s 
\le \deg b + 6d + 1.
$$
\end{proof}

\begin{corollary}\label{choice-of-iota}
In Theorem \ref{thm593} we can choose $\iota\le 6d+1 \le 6(n-l)+1$.
\end{corollary}
\begin{proof}
This is just \cite[Theorem 5.9.3 (b)]{gangolli1988harmonic} applied with $b=\id$ and observe that $\abso{\Sigma_r^+}\le \abso{\Sigma^+}\le \dim \mfn = n-l$.
\end{proof}

\begin{remark}
    We would like to say a bit more about the index $k=[\mfw:\mfw_0]$. In terms of $d$ (the number of indivisible roots), it can grow exponentially: if $\mfw_0=\{\id\}$ then $k=|\mfw|$ and is, for example, $(d+1)!$ if the root system is $A_d$. In our applications $\mfw_0$ will be the parabolic subgroup generated by all simple roots but one; however, in this case $k$ can also grow exponentially. An example is given again by the system $A_d$. Suppose that $d=2p$ is even and remove the $p$-th root; the remaining system is then reducible and $\mfw_0$ is the direct product of $A_p$ and $A_{p-1}$. We have $k = \dfrac{ (2p+1)!} { (p+1)!\, p! }$, and up to a constant this is $4^p/\sqrt p$. Our final estimate by $6d+1$ is thus indeed much lower.

\end{remark}


{\bf Estimates of $\kappa$.}
In Lemma 5.6.6, $\beta$ is defined as $\max_{1\le i\le q} |H_i|$, where $(H_i)$ is a basis of $\mfa_0\cap \Bracket{\mfg, \mfg}$ which extends to a dual basis of $\mfa$ with respect to the simple roots $(\alpha_i)$.
In the proof of Proposition 5.8.3, after the formula (5.8.15), the constants $c$ and $\eta>0$ are chosen so that $\mfa^+_0(\eta)$ is nonempty, $0<c<1$ and for every $t\in\R^q$
$$
|t_1 H_1+\dots +t_q H_q| \ge c\|t\|.
$$
Then one sets
$$
\kappa = \frac{ 2c\eta}{3q \beta (2k+1)}.
$$
All subsequent results relying on this Proposition use in fact the same value of $\kappa$.

Let us analyse it.
By (5.6.3), $\mfa^+_0 = \{ H\in \mfa_0: \alpha_i(H)>0, 1\le i\le q\}$, and by (5.8.1),
$$
\mfa_0^+(\eta) =  \{ H\in \mfa_0: \tau_0(H) > \eta|H| \},
$$
where $\tau_0(H) = \min\limits_{1\le i\le q} |\alpha_i(H)|$ by (5.6.3).

It follows that $\kappa$ depends on the root system only. In other words, for a given group it can be chosen once and be valid for any choice of $H_0$ or $\zeta$.


\section{Uniform estimates of kernels}\label{sec-main}

This is the central part of our article. We start with a general lemma concerning functions on a general group $G$ in class $\mathcal{H}$, and then we derive from it estimates for the kernels of a class of functions of Laplace operators, for $G$ semisimple.


\begin{lemma}\label{main lemma}
    Let $G$ be a group of class $\mathcal{H}$ and $R$ a radial positive function on $\mfa^*$. 
    Then 
    \begin{equation}\label{main inequality}
	\sup_{a\in \CA} \e^{\rho\,(\log a)} \int_{\mfa^*}R(\lambda)\abso{\varphi_\lambda(a)} \abso{\HCc(\lambda)}^{-1}d \lambda 
\le C  \int_{\mfa^*} R(\lambda) (1+|\lambda|)^{6d+\frac{n-l}{2}+1} d\lambda,
    \end{equation}
    where the constant $C>0$ depends on the group only.
Moreover,
$$
\limsup_{a\in \CA \atop |\log a|\to\infty} \e^{\rho\,(\log a)} \int_{\mfa^*}R(\lambda)\abso{\varphi_\lambda(a)} \abso{\HCc(\lambda)}^{-1}d \lambda 
\le C \int_{\mfa^*} R(\lambda) (1+|\lambda|)^{\frac{n-l}2} d\lambda.
$$

\end{lemma}
\begin{proof}
\textbf{Step 1.}
First, we deal with the simplest case $\abso{\Sigma_s^+}=1$. 
Throughout this step we assume that $\Sigma_s^+=\set{\alpha}$. Then $\Sigma^+=\set{\alpha}$ or $\Sigma^+=\set{\alpha,2\alpha}$, and $\Sigma=\set{\pm\alpha}$ or $\Sigma=\set{\pm \alpha,\pm 2\alpha}$. 
Since the argument will be similar and easier when $\Sigma^+=\set{\alpha}$, in the following proof, we only discuss the case $\Sigma^+=\set{\alpha,2\alpha}$. 
Under the assumption above we have
\[
\mfa^+=\set{H\in \mfa: \alpha(H)> 0},\quad
\Ca=\set{H\in \mfa: \alpha(H)\ge 0}.
\]

With the fact $s_\alpha=s_{-\alpha}=s_{2\alpha}$ we know that $\mfw=\set{\id, s_\alpha}$ and $\Lambda=\alpha\Z^+$. We rewrite the Harish-Chandra series of $\varphi_\lambda$ as follows: for all regular $\lambda\in \mfa^*$ and all $H\in \mfa^+$,
\begin{equation}\label{phi-lambda-rank-one}
\varphi_\lambda(\exp H)\e^{\rho(H)}
= \HCc(\lambda)\e^{i\lambda(H)}\sum_{m=0}^\infty\Gamma_{m\alpha}(\lambda) \e^{-m\alpha(H)}
+ \HCc(s_\alpha \lambda)\e^{i s_\alpha\lambda(H)}\sum_{m=0}^\infty \Gamma_{m\alpha}(s_\alpha\lambda) \e^{-m\alpha(H)}.
\end{equation}
By \cite[Theorem 4.5.4]{gangolli1988harmonic}, there exist constants $C>0$, $p\ge 
0$ such that $|\Gamma_{m\alpha}(\lambda)| \le C m^p$ for all $\lambda\in\mfa^*$.

Then we can estimate spherical functions as follows:
\begin{equation*}
    \begin{split}
		\abso{\varphi_\lambda(\exp H)}\e^{\rho(H)}
		&\le \abso{\HCc(\lambda)}\sum_{m=0}^\infty \e^{-m\alpha(H)} \bracket{\abso{\Gamma_{m\alpha}(s_\alpha\lambda)}+\abso{\Gamma_{m\alpha}(\lambda)}}\\
		&\le 2 \abso{\HCc(\lambda)}\sum_{m=0}^\infty \e^{-m\alpha(H)} m^p,
    \end{split}
\end{equation*}
the series converging for any $H\in\mfa^+$. 
Set 
\[
\mfa^+_1=\set{H\in \mfa^+: \alpha(H)> 1}.
\]
We obtain that for any regular $\lambda$ and $H\in \mfa_1^+$,
\[
\abso{\varphi_\lambda(\exp H)}\e^{\rho(H)} \le C'\abso{\HCc(\lambda)},
\]
with
\[
C' = 2 \sum_{m=0}^\infty e^{-m} m^p.
\]

For any $H\in \Ca\backslash \mfa_1^+=\set{H\in \Ca: \alpha(H)\le 1}$, in virtue of $\norm{\varphi_\lambda}_\infty=\varphi_\lambda(1)=1$,
\[
\abso{\varphi_\lambda(\exp H)}\e^{\rho(H)}
\le \e^{\rho(H)}\cdot 1
\le \e^{\frac{1}{2}m_\alpha+m_{2\alpha}} =: C''
< \infty.
\]
Thus, for every $H\in \Ca$ and regular $\lambda$, we have $|\varphi_\lambda(\exp H)|\e^{\rho(H)} \le C 
\max(1,|\HCc(\lambda)|)$,
with $C=\max(C', C'')$.

The set $\mfa^*\setminus \mfa^{*\prime}$ of irregular points has measure zero in $\mfa^*$ and does not influence integration. 
Together with \eqref{c-function},	
we now obtain
\[
\sup_{a\in \CA}\int_{\mfa^*}R(\lambda)\abso{\varphi_\lambda(a)}\e^{\rho(\log a)} \abso{\HCc(\lambda)}^{-1}d \lambda 
\le C \int_{\mfa^*}R(\lambda)\bracket{1+\abso{\lambda}}^{\frac {n-l}2}d\lambda,
\]
addressing both assertions of the theorem. 
	
\textbf{Step 2.}
We induct on $\abso{\Sigma_s^+}$. 

If $\abso{\Sigma_s^+}=1$, by the discussion in \textbf{Step 1}, the result follows.
Assume now the theorem holds 
for all Lie groups of class $\mathcal{H}$ whose simple system satisfies $\abso{\Sigma_s^+}\le r-1$.
We proced to consider the case $\abso{\Sigma_s^+}= r$.

We denote by $Z_{\mfg}$ the centre of $\mfg$ (which may be nontrivial), and set $\mfZ=\mfa\cap Z_{\mfg}$.	 
As a subspace of $\mfa$,
\[
\mfZ=\bigcap_{\alpha\in \Sigma^+}\ker \alpha=\bigcap_{\alpha\in \Sigma_s^+ } \ker\alpha.
\]
For each positive root $\alpha$ we can define the quotient linear functional $\bar\alpha: \mfa/\mfZ\rightarrow \R $ by
\[
\bar \alpha (\bar H):=\alpha(H),\  \bar H = H+\mfZ\in \mfa/ \mfZ.
\]
It is clear that $p(\bar H):=\max_{\alpha
\in \Sigma^+_s }\abso{\bar\alpha (\bar H)}$ is a norm on $\mfa/\mfZ$.
Moreover, it is equivalent to the quotient norm since $\mfa/\mfZ$ is finite-dimensional.
It follows that
we can find a constant $\zeta>0$ such that $p(\bar H)>\zeta \norm{\bar H}$ for any $H\notin \mfZ$.
Recall that the simple system is $\Sigma_s^+=\set{\alpha_1,\cdots,\alpha_r}$. 
For each $H\notin \mfZ$ there exists $\alpha_k\in \Sigma_s^+$ such that $p(\bar H)=\abso{\bar\alpha_k(\bar H)} > \zeta \norm{\bar H}$. 
By definition of the quotient norm, this implies that there exists $Y\in \mfZ$ such that $\abso{\alpha_k(H)}=\abso{\alpha_k(H+Y)}>\zeta \norm{H+Y}$.
Moreover, if $H\in \Ca$ then $\alpha_k(H)=\alpha_k(H+Y)> 0$ and also $H+Y \in \Ca$.
Define $\mfa_k^+=\set{X\in \Ca: \alpha_k(X)> \zeta \norm{X}}$, then $H=H+Y-Y\in \mfa_k^+ +\mfZ$ and we have just shown that 
\begin{equation}\label{cover}
    \Ca \subset \bigcup_{k=1}^r (\mfa^+_k+\mfZ)\cup \mfZ.
\end{equation}

We choose next a basis $\set{\bar H_1,\dots, \bar H_r}$ of $\mfa/\mfZ$ dual to $\set{\bar\alpha_1, \dots, \bar\alpha_r}$ and pick representatives ${H_k\in \mfa}$ of $\bar H_k$ such that $\alpha_k(H_k) > \zeta \norm{H_k}$. This ensures in particular that $\alpha_j(H_k)=0$ if $j\ne k$, and $H_k\in \mfa_k^+\subset \Ca$.


\textbf{Step 3.}
For any $H\in \mfZ,\ \alpha\in \Sigma^+$ and $Y\in \mfg_\alpha$, we get $\alpha(H)Y=\Bracket{H,Y}=0$. 
It implies $\rho(H)=0$ and $\delta(a)=\e^{-2\rho(H)}=1$.
With the fact $\norm{\varphi_\lambda}_\infty = 1$ we have
\begin{equation}\label{on_the_center}
\sup_{a\in \exp \mfZ}\int_{\mfa^*}R(\lambda)\abso{\varphi_\lambda(a)}\e^{\rho(\log a)} \abso{\HCc(\lambda)}^{-1}d \lambda 
\le \int_{\mfa^*}R(\lambda)\bracket{1+\abso{\lambda}}^{\frac{n-l}{2}} d\lambda .
\end{equation}

Next, estimates on $\mfa^+_k+\mfZ$ are the same as on $\mfa^+_k$, which is verified as follows.
By \cite[Chapter IV, Lemma 4.4]{helgason2022groups} we have for any $Y\in \mfZ$ and $H\in \Ca$, 
\begin{equation*}
    \begin{split}
        \varphi_\lambda(\exp H)
		&=\varphi_\lambda(\exp(-Y+Y+H))\\
		&=\int_K \e^{(-i\lambda+\rho)(A(k\exp Y))}\e^{(i\lambda+\rho)(A(k\exp(Y+H)))}dk\\
		&=\int_K \e^{(-i\lambda+\rho)(A(\exp Yk))}\e^{(i\lambda+\rho)(A(k\exp(Y+H)))}dk\\
		&=\int_K \e^{(-i\lambda+\rho)(Y)}\e^{(i\lambda+\rho)(A(k\exp(Y+H)))}dk\\
		&=\e^{-i\lambda(Y)}\int_K \e^{(i\lambda+\rho)(A(k\exp(Y+H)))}dk\\
        & = \e^{-i\lambda(Y)} \varphi_\lambda(\exp(Y+H)).
    \end{split}
\end{equation*}
Thus we obtain 
\begin{equation*}
    \abso{\varphi_\lambda(\exp(H+Y))}\e^{\rho(H+Y)}
    =\abso{\varphi_\lambda(\exp H)}\e^{\rho(H)},
\end{equation*}
that is,
\begin{equation}\label{relation on the quotient chamber}
\sup_{H\in \mfa^+_k+\mfZ} |\varphi_\lambda(\exp H)|\e^{\rho(H)}
 = \sup_{H\in \mfa^+_k} |\varphi_\lambda(\exp H)|\e^{\rho(H)}.
\end{equation}

\textbf{Step 4.}
For each $k\in \set{1,\dots,r}$, set $F_k = \{ \alpha_j: j\ne k\}$. This is exactly the set of simple roots of $G$ vanishing on $H_k$. Let $P_k=P_{F_k}$ be the parabolic subgroup associated with $F_k$, as described in Section \ref{Asymptotic behavior}. Denote by $G_k=M_{1F_k}$ the associated subgroup of $P_k$. The simple system on the Lie algebra $\mfg_k$ of $G_k$ is exactly $F_k$ and trivially $\abso{F_k}=r-1$. We write for convenience $\Sigma_k^+:=\Sigma_{F_k}^+$, $\Sigma_k:= \Sigma_{F_k}$, $\mfw_k:=\mfw_{F_k}$, $\HCc_k:=\HCc_{F_k}$, $\tau_k:=\tau_{F_k}$ and $\rho_k:= \rho_{F_k}= \frac{1}{2} \sum_{\alpha\in \Sigma_k^+} m_\alpha \alpha$. Then the conical region~\eqref{conical-region} corresponding to $H_k$ is
\[
A^+(H_k:\zeta)=\set{a\in \CA: \tau_k(\log a)> \zeta |\log a|}=\set{a\in \CA: \alpha_k(\log a) > \zeta |\log a|}.
\]
The sets $(\mfa_k^+)$ can be represented now in the following form: 
\[
\mfa_k^+ =\set{\log x: x\in A^+(H_k: \zeta)}.
\]
    
According to Notation \ref{thm586} and to Theorem \ref{thm593}, for every $k$ there exist constants $C_k(\zeta)$, $\iota_k$, $\kappa_k$ such that for any $a\in A^+(H_k: \zeta)$
\begin{equation*}
    \begin{split}
		\abso{\varphi_\lambda(a)}\e^{\rho(\log a)}
		&\le \e^{\rho_k(\log a)}\abso{\psi_\lambda(a)}+C_k(\zeta) \bracket{1+\abso{\lambda}}^{\iota_k} \e^{-\kappa_k \abso{\log a}}\\
		&= \abso{\abso{\mfw_k}^{-1} \sum_{s\in \mfw}\frac{\HCc(s\lambda)}{\HCc_k(s\lambda)} \theta_{s\lambda}(a)\e^{\rho_k(\log a)}}+C_k(\zeta)\bracket{1+\abso{\lambda}}^{\iota_k} \e^{-\kappa_k \abso{\log a}};
    \end{split}
\end{equation*}
by Corollary \ref{choice-of-iota}, we can assume $\iota_k = 6d+1$. 
Note also that $\zeta$ is the same for all $k$.
We denote then
\begin{align*}
  I_k'(a) &= \int_{\mfa^*}R(\lambda)\abso{\abso{\mfw_k}^{-1} \sum_{s\in \mfw} \frac{\HCc(s\lambda)}{\HCc_k(s\lambda)}\theta_{s\lambda}(a)\e^{\rho_k(\log a)} } \abso{\HCc(\lambda)}^{-1}d \lambda,
  \\
  I_k''(a) &= C_k(\zeta) \int_{\mfa^*}R(\lambda)\bracket{1+\abso{\lambda}}^{\iota_k} \e^{-\kappa_k |\log a|} \abso{\HCc(\lambda)}^{-1}d \lambda
\end{align*}
and
$$
		I_k(a) = I_k(a)'+ I_k(a)''.
$$

To estimate the main part $I_k(a)'$, recall first that $|\HCc(s\lambda)|=|\HCc(\lambda)|$ for each $s\in \mfw$, so that
\begin{align*}
I_k'(a)
		&\le \int_{\mfa^*}R(\lambda)\abso{\mfw_k}^{-1} \sum_{s\in \mfw} \frac{\abso{\HCc(s\lambda)}}{\abso{\HCc_k(s\lambda)}}\abso{\theta_{s\lambda}(a)}\e^{\rho_k(\log a)}\abso{\HCc(\lambda)}^{-1} d\lambda \\
		&=  \int_{\mfa^*}R(\lambda)\abso{\mfw_k}^{-1} \sum_{s\in \mfw} \abso{\HCc_k(s\lambda)}^{-1}\abso{\theta_{s\lambda}(a)}\e^{\rho_k(\log a)} d\lambda.  
\end{align*}
Next, every $s\in \mfw$ is an orthogonal transformation of $\mfa$, which implies
$R(\lambda)=R(\abso{\lambda})=R(\abso{s\lambda})=R(s\lambda)$. 
Changing the variables, we can now estimate the integral $I_k'(a)$ as follows:
\begin{equation*}
    \begin{split}
		I_k'(a)
		&\le |\mfw_k|^{-1} \sum_{s\in \mfw}\int_{\mfa^*}R(s\lambda)  \abso{\HCc_k(s\lambda)}^{-1}\abso{\theta_{s\lambda}(a)}\e^{\rho_k(\log a)} d\lambda \\
		&= \frac1{|\mfw_k|} \sum_{s\in \mfw} \e^{\rho_k(\log a)} \int_{\mfa^*}R(\lambda) \abso{\HCc_k(\lambda)}^{-1}\abso{\theta_{\lambda}(a)} d\lambda \\
		&= \frac{|\mfw|}{|\mfw_k|} \,\e^{\rho_k(\log a)} \int_{\mfa^*}R(\lambda)\abso{\theta_{\lambda}(a)} \abso{\HCc_k(\lambda)}^{-1}d\lambda.
    \end{split}
\end{equation*}
The simple roots of $G_k$ with respect to $\mfa$ are $\{\alpha_j: j\ne k\}$ and there are $r-1$ of them. The number $d_k$ of indivisible roots is smaller than $d$. We can apply thus the inductive hypothesis to get
    \begin{equation*}
        \begin{split}
            \sup_{a\in \CA} I_k'(a)
            \le C_k\int_{\mfa^*}R(\lambda)\bracket{1+\abso{\lambda}}^{6d_k+1+\frac{n_k-l}{2}} d\lambda
        \end{split}
    \end{equation*}
and
$$
\limsup_{|\log a|\to\infty}  I_k'(a) 
\le C_k\int_{\mfa^*}R(\lambda)\bracket{1+\abso{\lambda}}^{\frac{n_k-l}2} d\lambda,
$$
with a constant $C_k$ depending on the group $G_k$ only.
The remainder term $I_k''(a)$ is bounded by
$$
I_k''(a) \le C_k(\zeta)c_{G_k} \int_{\mfa^*}R(\lambda)(1+|\lambda|)^{6d+1+(n-l)/2} d \lambda\; \e^{-\kappa_k |\log a|}
; 
$$
it contributes to the uniform norm but not to the supper limit at $|\log a|\to\infty$.

    Thus, keeping the largest powers only, 
    \begin{equation*}
        \begin{split}
            \sup_{a\in A^+(H_k: \zeta)} I_k(a)
            &\le C_k' \int_{\mfa^*}R(\lambda)\bracket{1+|\lambda|}^{6d+\frac{n-l}{2}+1} d\lambda
        \end{split}
    \end{equation*}
and
$$
\limsup_{a\in A^+(H_k: \zeta) \atop |\log a|\to\infty}  I_k(a) 
\le C_k'\int_{\mfa^*}R(\lambda)\bracket{1+\abso{\lambda}}^{\frac{n-l}2} d\lambda
$$
with a constant $C_k'\ge1$.

With \eqref{cover}, it remains now to collect \eqref{on_the_center}, \eqref{relation on the quotient chamber} and the inequality above to get 
the statement of the theorem, with $C=\max\{ C'_k: 1\le k\le r\}$.
\end{proof}

From this moment we suppose that $G$ is a semisimple connected Lie group with finite center and $F$ a Borel function on $\R^+$.
We proceed 
to obtain uniform estimates for the kernel of $F(L)$, and by consequence pointwise bounds for the kernel of $F(\Delta_\rho)$. The following lemma shows that it is sufficient to estimate the values on the subset $\CA\subset G$ only.

\begin{lemma}\label{A+}
Let $k_\La$ be defined as in Section \ref{La-LB}, then
    \[
    \norm{k_\La}_\infty=\sup_{x\in G}\abso{k_\La(x)}\le \sup_{a\in \CA}\abso{k_\La(a)}.
    \]
\end{lemma}
\begin{proof}
	Every $x\in G$ can be decomposed as $x=k_1 a k_2$ where $a\in \CA$ and $k_1,\ k_2\in K$. 
	In virtue of \eqref{relation of kernels} and knowing that $k_\LB$ is $K$-biinvariant, we obtain
	\begin{equation*}
		\begin{split}
			k_\La(x)
			&=\delta^{-1/2}(k_1 ak_2)k_\LB(k_1ak_2)\\
			&=\delta^{-1/2}(ak_2)k_\LB(a)\\
			&=\delta^{-1/2}(ak_2)\delta^{1/2}(a)\delta^{-1/2}(a)k_\LB(a)\\
			&=\e ^{\rho(H(ak_2)-\log a)}k_\La(a).
		\end{split}
	\end{equation*}
	By \cite[Corollary 3.5.3]{gangolli1988harmonic} (see also \cite[Chapter 4, Lemma 6.5]{helgason2022groups}) we know that
	\[
	\abso{\rho(H(ak))}\le \rho(\log a)\quad \text{for all}\ a\in \CA,\ k\in K.
	\]
	Then we get 
	\[
	\e ^{\rho(H(ak_2)-\log a)}\le 1
	\] 
	and 
	\[
	\sup_{x\in G}\abso{k_\La(x)}\le \sup_{a\in \CA}\abso{k_\La(a)}.
	\]
	This completes the proof.
\end{proof}

Recall that $x^+$ denotes the $\CA$ component in the polar decomposition of $x\in G$.
\begin{theorem}\label{uniform estimate}
    Let $G$ be a semisimple connected Lie group with finite center and $F$ a Borel function on $\R^+$ such that 
    \begin{equation}\label{I_F}
I_F =    \int_{\mfa^*} \abso{F(\abso{\lambda}^2)} (1+|\lambda|)^{6d+n-l+1} d\lambda < \infty.
    \end{equation}
    Then $k_\La$ is uniformly bounded on $G$, and 
    \[
    \norm{k_\La}_\infty 
    =\sup_{x\in G}\abso{ \delta^{-\frac{1}{2}}(x)\int_{\mfa^*} F(\abso{\lambda}^2) \abso{\HCc (\lambda)}^{-2} \varphi_\lambda(x) d\lambda} 
    \le C I_F,
    \]
    where $C$ is a constant depending on $G$ only.
    Moreover,
    $$
    \limsup_{|\log x^+|\to\infty} |k_\La(x)| \le C \int_{\mfa^*} |F(|\lambda|^2)| (1+|\lambda|)^{n-l} d\lambda.
    $$
\end{theorem}
\begin{proof}
    By Lemma \ref{A+}, it is sufficient to estimate $|k_\La(a)|$
for $a\in \CA$.
    Since $G$ is semisimple it is also in class $\mathcal{H}$; we take 
    \[
    R(\lambda)=\abso{F(\abso{\lambda}^2)}\bracket{1+\abso{\lambda}}^{\frac{n-l}{2}}
    \]
    in Lemma \ref{main lemma} to obtain, using \eqref{c-function}, that
	\begin{equation*}
		\begin{split}
        \sup_{a\in \CA}\abso{k_\La(a)}
		&\le \sup_{a\in \CA}\int_{\mfa^*}\abso{F(\abso{\lambda}^2)}\e^{\rho(\log a)} \abso{\varphi_\lambda(a)}\abso{\HCc(\lambda)}^{-2} d\lambda\\
		&\le \sup_{a\in \CA}\int_{\mfa^*} R(\lambda) \,\e^{\rho(\log a)} \abso{\varphi_\lambda(a)}\abso{\HCc(\lambda)}^{-1} d\lambda\\
		&\le C\int_{\mfa^*} R(\lambda) (1+|\lambda|)^{6d+1+\frac{n-l}{2}}d \lambda\\
		&= C\int_{\mfa^*} |F(|\lambda|^2)| \bracket{1+|\lambda|}^{6d+n-l+1}d \lambda.
		\end{split}
	\end{equation*}
	The second inequality is obtained similarly.
\end{proof}

\begin{remark}
In the statement of the theorem, one can clearly replace \eqref{I_F} by    
    \begin{equation}\label{I_F'}
I_F =    \int_{\mfa^*} \abso{F(\abso{\lambda}^2)} (1+|\lambda|)^{7(n-l)+1} d\lambda < \infty.
    \end{equation}
\end{remark}

\begin{corollary}\label{main-LB}
    Let $G$ be a semisimple connected Lie group with finite center and $F$ a Borel function on $\R^+$ verifying \eqref{I_F} or \eqref{I_F'}.
    Then  for every $x\in G$
    \[ 
|k_\LB(x)|  \le CI_F\, e^{ -\rho( H(x) ) }, 
    \]
    where $C$ is a constant depending on $G$ only.
    Moreover,
    $$
    \limsup_{|\log x^+|\to\infty} |k_\LB(x)| e^{\rho( H(x) ) } \le C \int_{\mfa^*} |F(|\lambda|^2)| (1+|\lambda|)^{n-l} d\lambda.
    $$    
\end{corollary}

\begin{remark}
It follows from Theorem \ref{uniform estimate} that the operator $F(L)$ is bounded from $L^1(G,m_r)$, with respect to the {\it right} Haar measure, to $L^\infty(G)$. This occurs frequently, as soon as the integral \eqref{I_F} converges,  without any smoothness or positivity assumptions on the symbol.

One can also note that from $L^1(G,m_l)$ to $L^\infty(G)$ these operators are {\it never} bounded, unless $F$ is zero (this is a general fact about left convolution operators). 
\end{remark}

\begin{remark}  
From the known estimate \cite[Proposition 4.6.1, Theorem 4.6.4]{gangolli1988harmonic} 
of spherical functions
$$
|\varphi_\lambda(e^H)| \le C (1+\|H\|)^d e^{-\rho(H)},\ \text{for all}\ H\in \CA
$$
it is easy to obtain estimates for the kernel of the type $|k_{F(L)}(e^H)| \lesssim (1+\|H\|)^d$ and $|k_{F(\Delta_\rho)}(e^H)| \lesssim (1+\|H\|)^d e^{-\rho(H)}$. Our result is that it is possible to remove this polynomial factor.
In the analysis of the Laplace-Beltrami operator $\Delta$, a factor of this kind is often non-significant. But it becomes essential in concern with the operator $L$.
\end{remark}

\section{Lower bounds for oscillating functions in rank one}\label{lower_bounds}

Estimates of Section 4 involve the absolute value of the function $F$ and may seem unsatisfactory for oscillating functions of the type $F(x) = \exp(it\sqrt x) \,\psi(\sqrt x)$, where one would hope to see a decrease in $t$ as~$t\to\infty$. However, they are optimal even in this case, if $G$ has rank one. We show this in the present section. In higher rank the behaviour would be different, but it will be considered elsewhere.

Let us assume now that $G$ is a semisimple Lie group of real rank one.
Both $\mfa$ and $\mfa^*$ can be viewed then as the real line $\R$ and $\abso{\Sigma_s^+}=1$. Let $\alpha$ be the only positive root; we can assume that it acts at $H\in \mfa\simeq\R$ as simple multiplication, $\alpha(H) = \alpha H$.
For all $H\in \R^+$ and all regular $\lambda\in \R$, we have the Harish-Chandra decomposition \eqref{phi-lambda-rank-one} of $\varphi_\lambda(\exp H)\e^{\rho(H)}$.
By \cite[Theorem 4.5.4]{gangolli1988harmonic}, there exist constants $C_\alpha>0$, $p>0$ such that $|\Gamma_{m\alpha}(\lambda)| \le C_\alpha m^p$ for all $\lambda\in\mfa^*$. For $H$ such that $\alpha(H)>1$, we can then estimate
\begin{equation*}
    \begin{split}
         & \abso{\varphi_\lambda(\exp H)\e^{\rho(H)}- \HCc (\lambda)\e^{i\lambda(H)}-\HCc (s_\alpha\lambda)\e^{is_\alpha\lambda(H)}}\\
        &\quad \le C_\alpha \big( |\HCc(\lambda)| + |\HCc(s_\alpha \lambda)| \big) \sum_{m=1}^\infty m^p \e^{-m\alpha(H)}
        < C |\HCc(\lambda)| e^{-\alpha(H)}.
    \end{split}
\end{equation*}
Since $s_\alpha\lambda= -\lambda$ for all $\lambda\in \mfa^*$, we have
\[
    \begin{split}
        \abso{k_\La(\exp H)}
        &= \abso{\int_{\R} F(\abso{\lambda}^2) \varphi_\lambda(\exp H)\e^{\rho(H)} \abso{\HCc (\lambda)}^{-2}d \lambda}\\
        &\ge \abso{\int_{\R} F(\abso{\lambda}^2)\abso{\HCc (\lambda)}^{-2} \bracket{ \HCc (\lambda)\e^{i\lambda(H)} + \HCc (-\lambda)\e^{-i\lambda(H)}} d \lambda}\\ 
        &\quad - C \int_{\R} \abso{F(\abso{\lambda}^2)} \e^{-\alpha(H)} |\HCc (\lambda)|^{-1} d \lambda.
    \end{split}
\]

Now we set $F(x)=\e^{it\sqrt{x}}\psi(\sqrt{x})$ where $t>0$ and $\psi$ is a continuous function on $\R$ such that 
\begin{equation}\label{wavecondition}
    \begin{split}
        J_\psi = \int_0^\infty |\psi(x)| (1+x)^{n+6} dx < \infty.
    \end{split}
\end{equation}
This corresponds to \eqref{I_F} in rank one when $d=l=1$.
Denote by $k_t$ the kernel of the operator $F(L)$, then under the assumptions above on $H$ we have $k_t(\exp H)\ge \abso{I_1(t,H)}-I_2(H)$ where
\[
I_1(t, H)= \int_{\R} \e^{it\abso{\lambda}}\psi(\abso{\lambda}) \abso{\HCc (\lambda)}^{-2} \bracket{ \HCc (\lambda)\e^{i\lambda(H)} + \HCc (-\lambda)\e^{-i\lambda(H)}} d \lambda, 
\]
\[
I_2(H) = C \e^{-\alpha(H)} \int_{\R} \abso{\psi(\abso{\lambda})} |\HCc (\lambda)|^{-1} d \lambda.
\]
For the asymptotics of $k_t$ with large $H$, we only need to consider the $I_1(t,H)$, as $I_2(H)$ vanishes as $H\rightarrow +\infty $ under the condition (\ref{wavecondition}). Let us write $H=t+a$ with $a\in \R$. Using the conjugation property \eqref{c-conjugate} of the $\HCc$-function, we transform $I_1$ as follows:
\begin{align*}\label{int_psi_c}
        \frac12\, I_1(t,t+a)
        &=\int_{\R} \e^{it\abso{\lambda}}\psi(\abso{\lambda})\abso{\HCc (\lambda)}^{-2} \real \bracket{ \HCc (\lambda)\e^{i\lambda (t+a)}} d \lambda \\
        &=\int_0^\infty \e^{itx}\psi(x)\abso{\HCc(x)}^{-2} \real \bracket{\HCc(x)\e^{i(t+a)x}+\HCc(-x)\e^{-i(t+a)x} } d x\\
        &= 2\int_0^\infty \psi(x)\abso{\HCc(x)}^{-2} \e^{itx}\real \bracket{\HCc(x)\e^{i(t+a)x}} d x.
\end{align*}
After elementary calculations this takes form
\begin{align*}
        &= \int_0^\infty \psi(x)\abso{\HCc(x)}^{-2} \Big[ \real \HCc(x) \Big( (1+\cos(2tx)) \cos(ax) - \sin(2tx) \sin(ax)\\
        &\quad + i\sin(2tx) \cos(ax) -i (1-\cos(2tx)) \sin(ax) \Big)\\
        &\quad - \imaginary \HCc(x) \Big( \sin(2tx) \cos(ax) + (1+\cos(2tx)) \sin(ax) \\
        &\quad + i (1-\cos(2tx)) \cos(ax) + i \sin(2tx) \sin(ax) \Big) \Big] d x,
\end{align*}
which tends at $t\to+\infty$ to
\begin{align*}
        & \int_0^\infty \psi(x)\abso{\HCc(x)}^{-2} \Big[ \Big( \real \HCc(x) \cos(ax) - \imaginary \HCc(x) \sin(ax) \Big)\\
        &\quad -i \Big( \real \HCc(x) \sin(ax) + \imaginary \HCc(x) \cos(ax) \Big) \Big] d x\\
        &= \int_0^\infty \psi(x)\abso{\HCc(x)}^{-2} \Big[ \real \big( \HCc(x) e^{iax} \big) - i\imaginary \big( \HCc(x) e^{iax} \big) \Big] d x\\
        &=\int_0^\infty \psi(x)\abso{\HCc(x)}^{-2}\overline{\HCc(x) } e^{-iax} d x\\
        &=\int_0^\infty \psi(x)\HCc(x)^{-1} e^{-iax} d x .
\end{align*}
This is the Fourier transform at $a$ of the function $\tilde\psi$ which is equal to $\psi(x)\HCc(x)^{-1}$ if $x\ge0$ and 0 otherwise. Note that $\tilde\psi\in L^1$ since $\psi$ satisfies (\ref{wavecondition}).
For $\psi\not\equiv0$, this Fourier transform $\mF(\tilde\psi)$ is a nonzero continuous function vanishing at infinity, so its norm $\nu=\norm{\mF(\tilde\psi)}_\infty>0$ is attained at a point $a_0\in\R$. 
For $t\ge \alpha^{-1}-a_0$, we have $\alpha(H)=\alpha(a_0+t)\ge 1$, and the estimate above is applicable. Altogether, this implies that 
\[
\liminf_{t\rightarrow +\infty}\|k_t\|_\infty\ge \lim_{t\rightarrow +\infty}\abso{k_t(\exp (t+a_0))} \ge \lim_{t\rightarrow +\infty}(\abso{I_1(t,t+a_0)}-I_2(t+a_0))=\nu >0.
\]
In particular, this applies to the most common localization function $\psi(x)=\bracket{1+x^2}^{-\kappa}$ where $\kappa$ is a constant such that (\ref{wavecondition}) holds. 
To summarize the above, we proved the following theorem:
\begin{theorem}\label{lower-rank1}
    Let $G$ be a semisimple Lie group of real rank one. Denote by $k_t$ the kernel of $\e^{it\sqrt{L}}\psi(\sqrt{L})$ where $t>0$ and $\psi$ is a continuous function on $\R$ satisfying \eqref{wavecondition}.
    Then $\|k_t\|_\infty\le CJ_\psi$ for all $t$, and $\lim\limits_{t\rightarrow +\infty}\norm{k_t}_\infty \ge \nu$ where $\nu=\norm{\mF(\tilde\psi)}_\infty>0$ and
    \[
    \tilde\psi(x) = 
    \begin{cases}
    \psi(x)\HCc(x)^{-1} & \text{if}\ x\ge 0,\\
    0 & \text{if}\ x< 0.
    \end{cases}
    \]
    In particular, if $\psi(x)=\bracket{1+x^2}^{-\kappa}$, then the theorem applies if $\kappa>(n+7)/2$.
\end{theorem}

\addcontentsline{toc}{section}{References}

\bibliographystyle{amsplain}
\bibliography{pre-ref}

\end{document}